\documentclass[12pt]{amsart}
\usepackage{mathrsfs}
\usepackage{amsfonts}

\usepackage{amsmath, amsthm, amsfonts, amssymb}
\newtheorem{thm}{Theorem}[section]
\newtheorem{cor}[thm]{Corollary}
\newtheorem{lem}[thm]{Lemma}
\newtheorem{prop}[thm]{Proposition}
\newtheorem{claim}[thm]{Claim}
\newtheorem{example}[thm]{Example}

\theoremstyle{definition}

\def\Label#1{\label{#1}}

\numberwithin{equation}{section} \theoremstyle{remark}
\newtheorem{rem}{Remark}[section]
\def\<{\langle}
\def\>{\rangle}

\def\ra{\rightarrow}
\def\p{\partial}
\def\a{\alpha}
\def\owt{o_{wt}}

\def\l{{\lambda}}

\def\HH{{\mathbb H}}

\def\-{\overline}

\def\h{\hbox}
\def\d{\delta}

\def\d{\delta}
\def\d{\delta}

\def\a{\alpha}

\def\B{{\mathbb B}}
\def\CP{{\mathbb {CP}}}
\def\bP{{\mathbb {CP}}}
\def\CC{\mathbb C}
\def\bC{\mathbb C}
\def\RR{\mathbb  R}

\def\bH{\mathbb  H}

\def\M*{\wt{M^*}}

\def\-{\overline}

\def\h{\hbox}

\def\wt{\widetilde}
\def\ra{\rightarrow}

\def\d{\delta}

\def\a{\alpha}
\def\d{\delta}

\def\d{\delta}

\def\a{\alpha}

\def\im{\text{{\rm Im }}}

\def\beq{\begin{equation}}
\def\nneq{\end{equation}}
\def\beqn{\begin{eqnarray}}
\def\neqn{\end{eqnarray}}
\def\beqna{\begin{eqnarray*}}
\def\neqna{\end{eqnarray*}}
\def\bedis{\begin{displaymath}}
\def\nedis{\end{displaymath}}
\def\-{\overline}
\begin{document}

\bigskip

\title[Holomorphic Mappings between Hyperquadrics]{Holomorphic Mappings between Hyperquadrics with Small  Signature Difference}

\author[M. S. Baouendi, P. Ebenfelt, X. Huang]
{M. Salah Baouendi, Peter Ebenfelt, Xiaojun Huang}
\footnotetext{{\rm The authors are supported in part by DMS-0701070,
DMS-0701121 and DMS-0801056, respectively\newline}}
\address{M. S. Baouendi, P. Ebenfelt: Department of Mathematics, University of
California at San Diego, La Jolla, CA 92093-0112, USA}
\email{sbaouendi@ucsd.edu, pebenfel@math.ucsd.edu}
\address{X. Huang:
Department of Mathematics, Rutgers University - Hill Center for the
Mathematical Sciences 110 Frelinghuysen Rd., Piscataway, NJ
08854-8019,USA.}
 \email{huangx@math.rutgers.edu}

\begin{abstract}
In this paper, we  study  holomorphic mappings sending a
hyperquadric of signature $\ell$ in $\bC^n$ into a hyperquadric of
signature $\ell'$ in $\bC^N$. We show (Theorem \ref{main}) that if
the signature difference $\ell'-\ell$ is not too large, then the
mapping can be normalized by automorphisms of the target
hyperquadric to a particularly simple form and, in particular, the
image of the mapping is contained in a complex plane of a dimension
that depends only on $\ell$ and $\ell'$, and not on the target
dimension $N$. We also prove a Hopf Lemma type result (Theorem
\ref{main2}) for such mappings.
\end{abstract}


\vspace{3cm}
\maketitle


\section{Introduction}

In a recent  paper of the first and the third authors \cite{BH},  it
was shown that a holomorphic mapping $U\to \bC^N$, where $U$ is an
open connected subset of $\bC^n$ and  $N\geq n$, sending a piece of
a real hyperquadric with positive signature into a real hyperquadric
with the same signature either possess the super-rigidity and the CR
transversality (or Hopf lemma) properties or sends the whole open
neighborhood $U$ into the target hyperquadric. The super-rigidity
phenomenon obtained in \cite{BH} contrasts with the rigidity of
holomorphic mappings between  Heisenberg hypersurfaces (i.e.\
hyperquadrics with $0$-signature) in complex spaces of different
dimensions, which holds only when the difference in dimension is
small. (See \cite{HJ} for a survey on this matter.)
 The result obtained in \cite{BH}  is more along the lines of the behavior of holomorphic mappings between bounded symmetric domains of rank at least two (see \cite{Mok}).
In this paper, we  study  holomorphic mappings between hyperquadrics
with different signatures. We will show that both a suitably
reinterpreted rigidity phenomenon and a  weaker notion of the Hopf
lemma property hold when the difference between the signatures is
not too large; see Theorems \ref{main} and \ref{main2} for the
precise formulation.

To state our main result, we first recall some notation and
definitions. For $0\le \ell \le n-1$, we define the generalized Siegel
upper-half space
\begin{equation}\Label{snl}
{\mathbb S}^{n}_{\ell}\colon
=\big\{(z,w)=(z_1,\cdots,z_{n-1},w)\in {\CC}^n: \ \ w=u+iv,\\ \
v>-\sum_{j=1}^{\ell}|z_j|^2+ \sum_{j=\ell+1}^{n-1}|z_j|^2\big\},
\end{equation}
where the first sum is understood to be $0$ if $\ell=0$. The
boundary of ${\mathbb S }^{n}_{\ell}$ is the standard hyperquadric
\begin{equation}\Label{hnl}
{\mathbb H}^{n}_{\ell}:=\big\{(z,w)=(z_1,\cdots,z_{n-1}, w)\in
{\CC}^n: \ \ w=u+iv,
\\ v=-\sum_{j=1}^{\ell}|z_j|^2+
\sum_{j=\ell+1}^{n-1}|z_j|^2\big\}.
\end{equation}
 If  $0<\ell<n-1$, it is
well known that any CR function defined over a connected  open piece
$M$ of ${\mathbb H}^n_\ell$ extends to a holomorphic function in a
neighborhood of $M$ in ${\CC}^n$ (see e.g. \cite {BER}). We denote
by $\hbox{Aut}_{0}({\mathbb H}_\ell^n)$  the stability group of ${\mathbb
H}_\ell^n$ at 0, i.e.\ the group of
local biholomorphisms of ${\CC}^n$ sending $0$ to itself and a piece of ${\mathbb
H}_\ell^n$ near the origin into ${\mathbb
H}_\ell^n$. We point out that if  $0\leq \ell\leq (n-1)/2$, then $\ell$ is the
signature  of the hyperquadric ${ \mathbb H}^{n}_{\ell}$.  (By the signature of a
connected Levi nondegenerate hypersurface, we mean the mininum of the number of
positive and negative eigenvalues of a representative of the Levi form at any point.)
In what follows, we shall mainly consider this case.

Our main result in this paper concerns holomorphic mappings $F$
defined in an open connected neighborhood $U$ of $0$ in $\bC^n$,
valued in $\bC^N$, and sending $\mathbb H^n_\ell\cap U$ into
$\bH^N_{\ell'}$. Using the coordinates $(z,w)\in \bC^{n-1}\times
\bC$ and $(z^*,w^*)\in \bC^{N-1}\times \bC$, we shall write the
components of $F$ in the form $z^*=\widetilde f(z,w)$, $w^*=g(z,w)$.
We remark that it is easy to see that the derivative $\partial
g/\partial w(0)$ is a real number. Using the notation above, we can
now state our first main result as follows.

\begin{thm}\Label{main} Let $F$ be a holomorphic map from an open connected
neighborhood $U$ of $0 \in \mathbb{C}^n$ into $\mathbb{C}^N$ with
$1<n< N$ and $ F(0)=0$. Assume that $F$ maps
$\mathbb{H}^n_\ell\cap U$ into ${\mathbb{H}}^N_{\ell'}$, with $\ell \leq (n-1)/2$ and $ \ell'\leq
(N-1)/2.$ Then, the following hold:
\medskip

\noindent {\rm (a)} If $\partial g/\partial w(0)> 0$, then
$\ell\leq \ell'$ and $n-\ell\leq N-\ell'$. Moreover, if
$\ell'<2\ell$, then there is $\gamma \in Aut_0(\mathbb{H}^N_{\ell'})$
such that
\begin{equation}\Label{normgeq0}
\gamma \circ F  (z,w)=(z_1,\cdots,
z_\ell,\psi(z,w),z_{\ell+1},\cdots,z_{n-1}, \psi(z,w),0,\ldots,0,w),
\end{equation}
where $\psi=(\psi_1,\cdots,\psi_{\ell'-\ell})$ is
holomorphic near $0$ (with the understanding that this term is not present when $\ell'=\ell$). \medskip

\noindent {\rm (b)} If $\partial g/\partial w(0)< 0$, then
$\ell'\geq n-1-\ell$ and $N-1-\ell'\geq \ell$. Moreover, if $
\ell'<n-1$, then there is $\gamma \in Aut_0(\mathbb{H}^N_{\ell'})$
such that
\begin{equation}\Label{normleq0}
\gamma \circ F(z,w)=(z_{\ell+1},\cdots,z_{n-1},\psi(z,w), z_1,\cdots,
z_\ell,\psi(z,w),0,\ldots,0,-w),
\end{equation}
where $\psi=(\psi_1,\cdots,\psi_{\ell'-(n-1-\ell)})$ is
holomorphic near $0$ (with the understanding that this term is not present when $\ell'=n-1-\ell$).
\end{thm}

We would like to make a few remarks.

\begin{rem}\Label{mainrem1}  In the notation of Theorem \ref{main},
if $\partial g/\partial w(0)\neq 0$ and $\ell'<n-1-\ell$, then by
part (b) of the theorem, necessarily $\partial g/\partial w(0)> 0$.
Thus, if $\partial g/\partial w(0)\neq 0$, $\ell'<n-1-\ell$, and
$\ell'<2\ell$, then the conclusion  in Theorem \ref{main} (a) holds.
\end{rem}

Recall that a
holomorphic mapping $F\colon U\subset \bC^n\to \bC^N$ sending a real
hypersurface $M\subset \bC^n$ into a real hypersurface $M'\subset
\bC^N$ is said to be CR transversal to $M'$ at $p\in M$ if
$$
T^{1,0}_{F(p)} M'+dF(T^{1,0}_pM)=T^{1,0}_{F(p)}\bC^N.
$$
It is well known and not difficult to see that a holomorphic mapping
$F$ as in Theorem \ref{main} is CR transversal to $\mathbb
H^N_{\ell'}$ at $0$ if and only if $\partial g/\partial w(0)\neq 0$. Moreover, if
$\partial g/\partial w(0)\neq 0$, then $\partial g/\partial w(0)>0$ if and only if there is a small open neighborhood $V$ of 0 in $\bC^n$ such that $F(\mathbb S^n_{\ell}\cap V)\subset \mathbb S^N_{\ell'}$. Similarly, if
$\partial g/\partial w(0)\neq 0$, then $\partial g/\partial w(0)<0$ if and only if there is a small open neighborhood $V$ of 0 in $\bC^n$ such that $F(\mathbb S^n_{\ell}\cap V)\subset \bC^N\setminus \overline{\mathbb S^N_{\ell'}}$.

Also, recall that if $M'$ is a real-analytic hypersurface in
$\bC^N$, defined near a point $p'\in M'$ by a real-analytic defining
equation $\rho'(Z,\bar Z)=0$, then for $q'$ near $p'$ the Segre
variety $Q'_{q'}$ of $M'$ at $q'$ is the complex manifold defined by
the holomorphic equation $\rho'(Z,\bar q')=0$. In particular, the Segre variety
of
$\mathbb H^N_{\ell'}$ at $0$ is given by $Q'_{0}=\{(z^*,w^*)\colon
w^*=0\}$. Hence, in the notation of Theorem \ref{main}, $F$ sends $U$ into $Q'_0$ if and only if $g\equiv 0$. We now state our second main result of this paper.

\begin{thm}\Label{main2} Let $F$ be a holomorphic map from an open
neighborhood $U$ of $0 \in \mathbb{C}^n$ into $\mathbb{C}^N$ with
$1<n< N$ and $ F(0)=0$. Assume that $\mathbb{H}^n_\ell\cap
U$ is connected and $F$ maps $\mathbb{H}^n_\ell\cap
U$ into ${\mathbb{H}}^N_{\ell'}$, with $\ell \leq (n-1)/2$ and $
\ell'\leq (N-1)/2.$ Assume that   $\partial
g/\partial w(0)=0$. Then, the following hold:
\medskip

\noindent {\rm (a)} If there is a point $p\in \mathbb H^n_\ell\cap U$ and an open neighborhood $V$ of $p$ in $U$ such that
$F(\mathbb S^n_{\ell}\cap V)\subset \mathbb S^N_{\ell'}$ and $\ell'<2\ell$, then $g\equiv 0$.
\medskip

\noindent {\rm (b)} If there is a point $p\in \mathbb H^n_\ell\cap
U$ and an open neighborhood $V$ of $p$ in $U$ such that $F(\mathbb
S^n_{\ell}\cap V)\subset\bC^N\setminus \overline{ \mathbb
S^N_{\ell'}}$ and $\ell'<n-1$, then $g\equiv 0$.
\end{thm}

\begin{rem}\Label{gequiv0} {\rm By Lemma 4.1 of \cite{BH} and the above discussion, under the
assumption that $\ell'<n-1$ and $F(U)\not \subset \mathbb
H^N_{\ell'}$,  there must be a point $p\in \mathbb{H}^n_\ell\cap
U$  and an open neighborhood
$V$ of $p$ in $\bC^n$ such that either $F(\mathbb S^n_{\ell}\cap
V)\subset \mathbb S^N_{\ell'}$ or $F(\mathbb S^n_{\ell}\cap
V)\subset\bC^N\setminus \overline{ \mathbb S^N_{\ell'}}$. Also, when
 $F(U)\subset \mathbb H^N_{\ell'}$ we have $g\equiv 0$  (see Lemma 4.1 in \cite{BH} or p. 605 in
\cite{BERtrans}).
 We remark that
in \cite{BH} it was proved that if $\ell'=\ell$ and $\partial g/\partial w(0)=0$, then
the stronger conclusion $F(U)\subset \mathbb
H^N_{\ell'}$ holds. The latter conclusion does not always hold when
$\ell'>\ell$ as is illustrated by Example \ref{example1} below.}
\end{rem}

\begin{rem}
The hypotheses on the signatures $\ell$ and $\ell'$ in Theorems
\ref{main} and \ref{main2}  are sharp in the sense that the
conclusions fail when the strict inequalities are replaced by
equalities. This is illustrated by Example \ref{example2} below.
\end{rem}

As in \cite{BH},  Theorem \ref{main} has immediate applications to the study
of  proper holomorphic mappings between classical domains in
complex projective spaces.
For $0\le \ell<n$, denote by ${\mathbb B}^n_\ell$ the domain in
${\mathbb {CP}}^n$ given by
\begin{equation}
{\mathbb B}_{\ell}^{n}:=\{[z_0,\cdots,z_n]\in {\CP}^n:
|z_0|^2+\cdots+|z_{\ell}|^2>|z_{\ell +1}|^2+\cdots+|z_n|^2\}.
\end{equation}
Then it is well known that the Cayley transformation $\Psi_{n}\colon  \bC^n\to  {\CP}^n$ given by
\begin{equation}\Label{Cayley}\Psi_{n}(z,
w)\colon=[i+w, 2z, i-w]
\end{equation}
 biholomorphically  maps the generalized
Siegel upper-half space
 ${\mathbb S}^n_{\ell}$ and its boundary ${\HH}^n_{\ell}$ into
${\B}^n_{\ell}\setminus \{[z_0,\cdots,z_{n}]:\ z_0+z_n =0\}$ and
$\p{\B}^n_{\ell}\setminus \{[z_0,\cdots,z_{n}]:\ z_0+z_n =0\}$,
respectively.

For
$ 0\le k \le m$, let $E_{(k,m)}$ denote the $m\times m$ diagonal
matrix with its first $k$ diagonal elements $-1$ and the rest $+1$,
and define
\begin{equation}
U(n+1,\ell+1) =\big\{
A\in \h{GL}(n+1,{\CC}),\ A E_{(\ell+1,n+1)}\-{A}^t=E_{(\ell+1,n+1)}
\big\}.
\end{equation}
In what follows, we will regard $U(n+1,\ell+1)$ as a
subgroup of the automorphism group of ${\CP}^n$ by identifying an
element $A$ in $U(n+1,\ell+1)$ with the holomorphic linear
map $\sigma\in \h{Aut}({\CP}^n)$ defined by
$\sigma([z_0,\cdots,z_n])=[z_0,\cdots,z_n] A.$
  Then, it is  well known that, with this identification, we have $U(n+1,\ell+1)
=\h{Aut}({\B}_{\ell}^{n})$, transitively acting on ${\B}_{\ell}^{n}$
(see e.g. \cite{CM}, $\S 1$)).

By repeating the arguments in the beginning of the proof of Theorem 1.1 in \cite{BH}, the following corollary can easily be deduced from Theorem \ref{main} above.

\begin{cor} \Label{maincor} Let $\ell\le (n-1)/2$ and $ \ell'\le (N-1)/2$. Let $F$
be a  holomorphic map from ${\B}^{n}_{\ell}$ into $\bP^N$ with
$1<n\leq N$. Then, the following hold:
\medskip

\noindent {\rm (a)} If $ \ell'< 2\ell$ and $F$ is  proper  from
${\B}^{n}_{\ell}$ into  ${\B}^{N}_{\ell'}$, then $F({\B}^n_\ell)$ is
contained in a linear projective subspace of ${\CP}^N$ of dimension
$n+\ell'-\ell.$
\medskip

\noindent {\rm (b)} If $ \ell'< n-1$ and $F$ is  proper  from
${\B}^{n}_{\ell}$ into ${\B}^{N}_{N-1-\ell'}$, then $F({\B}^n_\ell)$
is contained in a linear projective subspace of ${\CP}^N$ of
dimension $\ell'+\ell+1.$
\end{cor}

\section{Notation, Definitions, and Two Basic Lemmas}

We shall use the notation and definitions introduced in
\cite{BH}, which we recall here for the reader's convenience.
 Let $ {\HH}^n_{\ell}\subset{\CC}^n$ and
${\HH}^N_{\ell'}\subset{\CC}^N$ be the standard hyperquadrics
defined by \eqref{hnl}:
\begin{equation*}
\begin{split} {\HH}^n_{\ell}:&=\big\{(z,w=u+iv)\in {\CC}^{n},\ v=\h{Im}\,
w=\sum_{j=1}^{n-1}\delta_{j,\ell}|z_j|^2\big\}; \\
{\HH}^N_{\ell'}:&=\big\{(z^*,w^*=u^*+iv^*)\in {\CC}^{N}, \
v^*=\sum_{j=1}^{N-1}\delta_{j,\ell'}|z_j^*|^2\big\}. \end{split}
\end{equation*}
     Here and in what follows, we denote by
${\delta}_{j,\ell}$  the symbol which takes value $-1$ when $1\le
j\le \ell$ and $1$ otherwise. For $\ell'\ge \ell$ and $N\ge
n>\ell-1$, we define
$${\HH}^N_{\ell, \ell', n}:=\big\{(z^*,w^*)\in {\CC}^{N}, \
\h{Im}\,w^*=\sum_{j=1}^{N-1}\delta_{j,\ell, \ell',n}|z_j^*|^2\big\}.
\leqno (2.2)$$
     with $\delta_{j,\ell,\ell',n}=-1$ for $j\le \ell$ or $n\le j\le
n+\ell'-\ell-1$, and $\delta_{j,\ell,\ell',n}=1$ otherwise.
When $\ell'>\ell$, ${\HH}^N_{\ell'}$ is
biholomorphically equivalent to ${\HH}^N_{\ell, \ell', n}$ by the
     linear map
\begin{align*}{\tag{2.3}\sigma_{\ell,\ell',n}(z^*,w^*):&=\\
(z^*_1,\cdots,z^*_{\ell}&,z^*_{\ell'+1},\cdots
,z^*_{n-1},z^*_{\ell+1}, \cdots,z_{\ell'}^*, z_{n}^*,\cdots,
z^*_{N-1},w^*).}\end{align*}
Write $\displaystyle
L_j=2i\delta_{j,\ell}\-{z_j}\frac{\partial}{\partial w}+
\frac{\partial}{\partial z_j}$  for $j=1,\cdots, n-1$ and
$T={\displaystyle \frac{\partial} {\partial u}}$. Then
$L_1,\cdots,L_{n-1}$ form a global basis for the complex
tangent bundle $\h{T}^{(1,0)}{\HH}^n_\ell$, and
$T$ is a tangent vector field of ${\HH}^n_\ell$ transversal
     to $T^{(1,0)}{\HH}_\ell^n\oplus T^{(0,1)}{\HH}_\ell^n$.
     Parameterize ${\HH}_\ell^n$ by  the map
$(z,\-{z},u)\mapsto (z,u+i\sum_{j=1}^{n-1}\delta_{j,\ell}|z_j|^2)$.
In what follows, we will assign the weights of $z$ and $u$ to be $1$
and $2$, respectively. For a non-negative integer $m$, a function
$h(z,\-{z},u)$ defined in a small neighborhood  $M$ of $0$ in
${\HH}^n_\ell$  is said to be
     $\owt (m)$, if
${h(tz,t\-{z},t^2u)/|t|^{m}}\ra 0$ uniformly for $(z,u)$ on any
     compact subset of $M$ for $t\in {\RR}\, ,t\ra 0$. (In this case,
we write $h=\owt (m)$). By convention, we write
     $h=\owt (0)$ if $h(z,\-{z}, u)\ra 0$ as $(z,\-{z},u)\ra 0$. For a
smooth function $h(z,\-{z},u)$
     defined in $M$, we denote by
$h^{(k)}(z,\-{z},u)$  the sum of
     terms of weighted degree $k$ in the Taylor
     expansion of $h$ at $0$.  We also denote by
$h^{(k)}(z,\-{z},u)$  a weighted
     homogeneous polynomial of weighted degree
$k$, (even if there is no specified function $h$).  When
$h^{(k)}(z,\-{z},u)$ extends to a weighted  holomorphic polynomial
of weighted degree $k$, we write it as
      $h^{(k)}(z,w)$, or
$h^{(k)}(z)$ if it depends only on $z$. Here again, $z$ has weight
$1$ and $w$ has weight $2$.

For a sufficiently smooth
function $h=h(x_1,\ldots, x_m)$ defined in an open subset of $\bC^m$ and any multiple index
$\alpha=(\alpha_1,\cdots,\alpha_{m})$, we write
$$D^\a_x h=\frac{\partial^{\a_1+\cdots+\a_m} h}{\partial
x_1^{\a_1}\cdots\partial x_{m}^{\a_m}}.$$
For two $m$-tuples $x=(x_1,\cdots,x_m), y=(y_1,\cdots,y_m)$ of
complex numbers, we write
$$\langle x,y\rangle_{\ell}=\sum_{j=1}^{m}\delta_{j,\ell}x_jy_j,\
\ \h{and} \ \ \ |x|_\ell^2=\langle x,x\rangle _{\ell}.$$
For $\ell'\ge \ell$ and $\ell-1\le n\le m$, we write $\langle x,y\rangle _{\ell,
\ell',n}=\sum_{j=1}^{m}\delta_{j,\ell,\ell',n}x_jy_j$.

For the proof of Theorems \ref{main} and \ref{main2}, it will be more convenient to
assume that the map $F$ sends $M={\HH}^n_\ell\cap U$ into
${\HH}^{N}_{\ell,\ell',n}$, with the conclusions modified
accordingly. The proof of the theorem is based on an induction
argument,  in the spirit of the Chern-Moser theory, using the
weighted expansion of the basic equations describing the inclusion
$F(U\cap \mathbb{H}^n_\ell)\subset \mathbb{H}^N_{\ell,\ell', n}$.
The method used here is largely motivated by the work in Huang
\cite{Hu1}, Ebenfelt-Huang-Zaitsev \cite{EHZ2}, and Baouendi-Huang
\cite{BH}.

The following two basic lemmas, also stated and used in \cite{BEH}, will
be crucial in the proof of Theorem \ref{main}.

\begin{lem}\Label{LemmaI}
 Let $k,\ell,n$ be nonnegative integers such $1\leq k\leq n-2$. Assume that
$a_1,\ldots, a_k$, $b_1,\ldots, b_k$ are germs at $0\in \bC^{n-1}$ of
holomorphic functions such that
\begin{equation}\Label{e:basiceq}
\sum_{i=1}^ka_i(z) \overline{b_i(\xi)}= A(z,\bar \xi)\langle z,\bar
\xi\rangle _{\ell},
\end{equation}
where $A(z,\bar \xi)$ is a germ at $0\in \bC^{n-1}\times\bC^{n-1}$ of a
holomorphic function in $(z,\bar\xi)$. Then $A(z,\bar \xi)\equiv 0$.
\end{lem}

Lemma \ref{LemmaI} was proved in  \cite{Hu1}  (see Lemma 3.2, \cite{Hu1}). In that paper, the lemma is stated only for $\ell=0$, but the
proof for $\ell>0$ is identical (see also  Lemma 2.1 in \cite{EHZ2}).
Lemma \ref{LemmaI} was also a crucial
tool in the papers \cite{Hu1}, \cite{EHZ1}, \cite{EHZ2}.
 We shall also need the
following corollary, which follows from repeated use of Lemma
\ref{LemmaI}.

\begin{cor}\Label{CorI}
 Let $\ell,n,k_0, k_1,\ldots,k_r$ be nonnegative integers such $1\leq k_j\leq n-2$ for
 $j=0,\ldots r$. Assume that
$a_i^j$, $b_i^j$ for $j=0,\ldots, r$ and $i=1,\ldots, k_j$ are germs
at $0\in \bC^{n-1}$ of holomorphic functions such that
\begin{equation}\Label{e:basiceq3}
\sum_{j=0}^r \left(\sum_{i=1}^{k_j}a^j_i(z)
\overline{b^j_i(\xi)}\right)\langle z,\bar \xi\rangle_\ell^j=
A(z,\bar \xi)\langle z,\bar \xi\rangle _{\ell}^{r+1},
\end{equation}
where $A(z,\bar \xi)$ is a germ at $0\in \bC^{n-1}\times\bC^{n-1}$ of a
holomorphic function in $(z,\bar\xi)$. Then
\begin{equation}
\sum_{i=1}^{k_j}a^j_i(z) \overline{b^j_i(\xi)}=0, \quad j=0,\ldots,
r,
\end{equation}
and $A(z,\bar \xi)\equiv 0$.
\end{cor}

The second lemma that we shall need is the following.

\begin{lem}\Label{LemmaII}
Let $k,\ell,m,n$ be nonnegative integers such that
$k<l\leq (n-1)/2$ and $k\leq m$. Assume that $a_1,\ldots, a_k, b_1,\ldots, b_m$ are
germs at $0\in \bC^{n-1}$ of holomorphic functions such that
\begin{equation}\Label{e:basiceq2}
-\sum_{i=1}^k|a_i(z)|^2+\sum_{j=1}^m|b_j(z)|^2=
A(z,\bar z)|z|_\ell^2,
\end{equation}
where $A(z,\bar\xi)$ is a germ at $0\in \bC^{n-1}\times\bC^{n-1}$ of a
holomorphic function of $(z,\bar\xi)$. Then $A(z,\bar \xi)\equiv 0$
and $(b_1,\ldots, b_m)=(a_1,\ldots,a_k)\cdot \mathcal {U}$, where
$\mathcal {U}$ is a constant $(k\times m)$-matrix such that
$\mathcal {U}\cdot \-{\mathcal{U}}^{t}=Id_{k\times k}$.
\end{lem}

The proof of the first part of Lemma \ref{LemmaII},
$A(z,\bar\xi)\equiv0$, can be obtained from Lemma 4.1 of \cite{BH}
(with $\ell'=\ell$ and after a direct application of Lemma 2.1 of
\cite{BH}). This part of the lemma also follows in a straightforward
way from Theorem 5.7 in the subsequent work \cite{BERtrans}. The
last conclusion follows directly from a lemma by D'Angelo \cite{DA}.

\section{ Proof of Theorem $\ref{main}$}\Label{mainsec}

We shall first prove part (a) of Theorem \ref{main} and hence, in
addition to the main hypotheses in the theorem, we assume also that
$\partial g/\partial w(0)>0$. We use the notation introduced in $\S
2$ together with the basic set-up in \cite{BH}. First, we note that
the conclusions $\ell\leq \ell' $  and $n-\ell\leq N-\ell'$ follow
immediately by counting the number of negative and positive
eigenvalues on both sides in equation (2.6) in \cite{BH} (as in
Lemma 2.1 (a) in \cite{BH}).

In what follows, we identify $\mathbb H^N_{\ell'}$ with $\mathbb H^N_{\ell,\ell',n}$ as explained in Section 2. The fact that the mapping $F$ sends $\mathbb H^n_\ell$ into $\mathbb H^N_{\ell,\ell',n}$ means that the following basic equation for $F=(f,\varphi,g)=(f_1,\ldots, f_{n-1},\varphi_1,\ldots, \varphi_{N-n}, g)$  holds:
\begin{align}
\im g=\langle f,\-f \rangle_\ell+\langle \varphi,\-\varphi
\rangle_{\tau}\ \mbox{when}\ \im w=\langle z,\-z \rangle_\ell.
\label{A1}
\end{align}
Here and in what follows, we use the notation $\tau:=\ell'-\ell$. As
in the introduction, we shall also use the notation $\widetilde
f:=(f,\varphi)$. By the first part of Lemma 2.2 in \cite{BH}, we can
assume, without loss of generality, that $F$ has the following
normalization $F=(f,\varphi,g)$:
\begin{equation}\Label{norm}
\begin{aligned}
&f(z,w)=z+\frac{i}{2}a^{(1)}(z)w+o_{wt}(3),\\
&\varphi(z,w)=\varphi^{(2)}(z)+o_{wt}(2),\\
&g(z,w)=w+o_{wt}(4)
\end{aligned}
\end{equation}
with
\begin{align}\Label{SFF}
\left\langle a^{(1)}(z),\-z\right\rangle_\ell|z|^2_\ell=\left\langle
\varphi^{(2)}(z),\overline{\varphi^{(2)}(z)}\right\rangle_{\tau}.
\end{align}
 We remark that to achieve the normalization \eqref{norm} above,
it suffices to compose the original map $F$ from the left with
$\tilde{\sigma}\in Aut_0(\mathbb{H}^N_{\ell,\ell'})$ (see Lemma 2.2
in \cite{BH}). Since $\tau=\ell'-\ell<\ell$ by the assumption in part (a) of
the theorem, it follows from \eqref{SFF} and Lemma \ref{LemmaII}
that $\left\langle a^{(1)}(z),\-z\right\rangle_\ell\equiv 0$. Thus,
$a^{(1)}(z)\equiv 0$ (and hence $f^{(3)}\equiv 0$) and
\begin{equation}\Label{eqn: 001}
\left\langle
\varphi^{(2)}(z),\overline{\varphi^{(2)}(z)}\right\rangle_{\tau}\equiv0
\end{equation}

Assume that we have shown
$g^{(t)}\equiv 0$, $f^{(t-1)}\equiv 0$ for $3\leq t<s$. Observe that we have shown this for $s=5$.
Then collecting terms of weighted degree $s$ in (\ref{A1}), we obtain
\begin{multline}\Label{A2}
\im \left\{g^{(s)}(z,w)-2i\langle
\-z,f^{(s-1)}(z,w) \rangle_\ell\right\}=\\
 \sum\limits_{s_1+s_2=s}\langle
\varphi^{(s_1)}(z,w),\overline{\varphi^{(s_2)}(z,w)}\rangle_{\tau},\ \text{when } w=u+i\langle z,\bar z \rangle_\ell.
\end{multline}
For convenience, we shall use the notation \begin{equation}
\mathscr{L}(p,q)(z,\bar z,u):=\im \left\{q(z,w)-2i\langle
\-z,p(z,w) \rangle_\ell\right\}\bigg|_{w=u+i\langle z,\bar z\rangle_\ell},
\end{equation}
where $p(z,w)=(p_1(z,w),\ldots, p_{n-1}(z,w))$ and $q(z,w)$ are holomorphic polynomials. Thus, the equation \eqref{A2} can be written
\begin{equation}\Label{A22}
\mathscr{L}(f^{(s-1)},g^{(s)})(z,\bar z,u)=\sum\limits_{s_1+s_2=s}\langle
\varphi^{(s_1)}(z,w),\overline{\varphi^{(s_2)}(z,w)}\rangle_\tau\bigg|_{w=u+i\langle z,\bar z\rangle_\ell}
\end{equation}

The main step in the proof of Theorem \ref{main} is an induction based on the following result.

\begin{prop}\Label{mainprop} Let $F=(f,\varphi,g)$ be any  normalized map as in \eqref{norm} sending an open piece of  ${\HH}^n_\ell$ near the
origin into ${\HH}^N_{\ell,\ell'}$. Assume that for  all $3\leq t\leq 2(s^{\ast}-1)$
\begin{equation}\Label{induc}
f^{(t-1)}\equiv0,\ g^{(t)}\equiv 0,\
\langle \varphi^{(s_1)},\overline{\varphi^{(s_2)}} \rangle_{\tau}\equiv
0,\quad \forall(s_1,s_2)\colon s_1+s_2=t.
\end{equation}
Then \eqref{induc} holds also for $t=2(s^{\ast}-1)+1$ and
$t=2s^{\ast}$ for any such an $F$.
\end{prop}

We point out that we have already proved that \eqref{induc} holds for all $3\leq t\leq 4$. Hence, once the proposition has been proved, we conclude by induction that \eqref{induc} holds for all $t \geq 3$. For the proof of Proposition \ref{mainprop}, we shall need some notation and results from \cite{EHZ2}.

Given a real-valued power series $A(z,\bar z,w,\bar w)$, we shall use the expansion
\begin{equation}\Label{expand-s}
A(z,\bar z,w,\bar w)=\sum_{\mu,\nu,\gamma,\delta}
A_{\mu\nu\gamma\delta}(z,\bar z)w^\gamma \bar w^\delta,
\quad (z,w)\in\bC^{n-1}\times\bC,
\end{equation}
where $A_{\mu\nu\gamma\delta}(z,\bar z)$
is a bihomogeneous polynomial in $(z,\bar z)$ of bidegree $(\mu,\nu)$
for every $(\mu,\nu,\gamma,\delta)\in \mathbb Z_+^4$. We recall from \cite{EHZ2} that $A(z,\bar z,w,\bar w)$ is said to belong to the class $\tilde {\mathcal S}_k$, where $k$ is a positive integer, if $A$ vanishes at least to order 2 at $0$ and for every $(\mu,\nu,\gamma,\delta)\in \mathbb Z_+^4$ we have
\begin{equation}
A_{\mu\nu\gamma\delta}(z,\bar z)w^\gamma\bar w^\delta=\sum_{j=1}^k p_j(z,w)\overline{q_j(z,w)},
\end{equation}
where the $p_j$ and $q_j$ are homogeneous holomorphic polynomials of the appropriate degrees. We shall use the following result, which is direct consequence of Theorem 2.2 and Lemma 4.2 in \cite{EHZ2}.

\begin{thm}\Label{mainEHZ2} Let $A(z,\bar z,w,\bar w)$ be a real-valued weighted homogeneous polynomial  of degree $s\geq 5$, and assume that $A\in \tilde {\mathcal S}_{n-2}$. If $p(z,w)=(p_1(z,w),\ldots, p_{n-1}(z,w))$ and $q(z,w)$ are weighted homogeneous holomorphic polynomials of degree $s-1$ and $s$, respectively, such that
\begin{equation}\Label{CMeq}
\mathscr{L}(p,q)(z,\bar z,u)=A(z,\bar z,w,\bar w)\big|_{w=u+i\langle z,\bar z\rangle_\ell},
\end{equation}
then $p\equiv 0$, $q\equiv 0$, and $A\equiv 0$.
\end{thm}

We shall now give the proof of Proposition \ref{mainprop}.

\begin{proof}[Proof of Proposition $ \ref{mainprop}$] We shall first prove that \eqref{induc} holds for $t=2(s^*-1)+1=2s^*-1$ with $s^*\geq 3$. As mentioned above, the hypotheses in the proposition imply that \eqref{A22} holds with $s=t$. We note that the hypotheses also imply that $\langle \varphi^{(s_1)},\overline{\varphi^{(s_1)}}\rangle_\tau\equiv 0$ for $2\leq s_1\leq s^*-1$.  By a lemma of D'Angelo \cite{DA}, we conclude that there are constants $a^{s_1}_{jk}$ such that
\begin{equation}\Label{phiids}
\varphi^{(s_1)}_j\equiv \sum_{k=1}^{\tau}a^{s_1}_{jk}\varphi^{(s_1)}_k\ \
\mbox{for}\ \ j=\tau+1,\cdots,N-n,\ 2\leq s_1\leq s^*-1.
\end{equation}
Now, if $s_1+s_2=2s^*-1$, then $\min(s_1,s_2)\leq s^*-1$. If, say, $s_1=\min(s_1,s_2)\leq s^*-1$ (the case where $s_2=\min(s_1,s_2)$ is completely analogous and left to the reader), then it follows from the identities in \eqref{phiids} that
\begin{equation}
\begin{aligned}
\langle \varphi^{(s_1)},\overline{\varphi^{(s_2)}}\rangle_\tau &\equiv \sum_{j=1}^{N-n}\delta_{j\tau} \varphi^{(s_1)}_j\overline{\varphi_j^{(s_2)}}\\
&\equiv \sum_{j=1}^\tau\varphi^{(s_1)}_j \overline{q^{(s_2)}_j},
\end{aligned}
\end{equation}
where
\begin{equation}
q^{(s_2)}_j(z,w):=\delta_{j\tau}\varphi_j^{(s_2)}(z,w)+\sum_{k=\tau+1}^{N-n}\delta_{k\tau}a^{s_1}_{k j}\varphi_k^{(s_2)}(z,w).
\end{equation}
Since $\ell'<2\ell\leq n-1$ and $\tau=\ell'-\ell$, it follows that $\tau\leq n-2$. We conclude that
$$
A(z,\bar z,w,\bar w):=\sum\limits_{s_1+s_2=2s^*-1}\langle
\varphi^{(s_1)}(z,w),\overline{\varphi^{(s_2)}(z,w)}\rangle_\tau
$$
belongs to $\tilde{\mathcal S}_{n-2}$. It follows from \eqref{A22} and Theorem \ref{mainEHZ2} that $f^{(s-1)}\equiv 0$, $g^{(s)}\equiv 0$ with $s=2s^*-1$, and $A\equiv 0$. By the definition of $A$, we conclude that $\langle
\varphi^{(s_1)},\overline{\varphi^{(s_2)}}\rangle_\tau\equiv 0$ for $s_1+s_2=2s^*-1$. This
completes the proof of Proposition \ref{mainprop} for $t=2(s^*-1)+1=2s^*-1$.

 It remains to prove
\eqref{induc} for $t=2s^*\geq 6$. Complexifying (\ref{A2}) with
$s=2s^{\ast}$, we get
\begin{multline}\Label{A4}
g^{(2s^{\ast})}(z,w)-\overline{g^{(2s^{\ast})}(\xi,\eta)}-2i\left\langle
\-\xi,f^{(2s^{\ast}-1)}(z,w) \right\rangle_\ell-2i\left\langle
z,\overline{f^{(2s^{\ast}-1)}(\xi,\eta)} \right\rangle_\ell\\
= 2i\sum\limits_{k}\left\langle
\varphi^{(k)}(z,w),\overline{\varphi^{(2s^{\ast}-k)}(\xi,\eta)}
\right\rangle_\tau,\ \ \text{{\rm when }} w=\-\eta+2i\langle z,\-\xi
\rangle_{\ell}.
\end{multline}
Let us denote by ${\mathcal {L}}_j=\frac{\partial}{\partial
z_j}+2i\delta_{j\ell}\-\xi_j\frac{\partial}{\partial w}$ for
$j=1,\ldots ,n-1$. Observe that ${\mathcal {L}}_j$ is tangent to the
complex hypersurface defined by $w=\-\eta+2i\langle z,\-\xi
\rangle_{\ell}$. As a first step towards finishing the induction
step for $s=2s^*$ in the proof of Proposition \ref{mainprop}, we
shall prove that
\begin{equation}\Label{mainlem}
f^{(2s^*-2)}(z,w)\equiv 0,\quad g^{(2s^*)}(z,w)\equiv 0.
\end{equation}
We begin by establishing the following preliminary claim.

\begin{claim}\Label{mainclaim} We have
$f_j^{(2s^{\ast}-1)}(z,w) =a_j^{(1)}(z)w^{s^*-1}$,
$g^{(2s^{\ast})}(z,w)=d^{(0)}w^{s^*}$, and
\begin{equation}\Label{rel0}
a^{(1)}_j(z)=d^{(0)}z_j, \end{equation}
for $j=1,\ldots, n-1$.
\end{claim}

\begin{proof} [Proof of Claim $\ref{mainclaim}$] Applying ${\mathcal L}_j$ to (\ref{A4}),
we obtain for $w=\-\eta+2i\langle z,\-\xi \rangle_{\ell}$
\begin{multline}\Label{Leq}
{\mathcal L}_j\left(g^{(2s^{\ast})}(z,w)\right)-2i\left\langle
\-\xi,{\mathcal L}_j(f^{(2s^{\ast}-1)}(z,w))
\right\rangle_\ell-2i\overline{f_j(\xi,\eta)}\delta_{jl}\\=2i\sum\limits_{k}
\left\langle {\mathcal
{L}}_j\varphi^{(k)}(z,w),\overline{\varphi^{(2s^{\ast}-k)}(\xi,\eta)}
\right\rangle_{\tau}.
\end{multline}
Let us write
\begin{equation}\Label{expans}
f_j^{(s-1)}(z,w)=\sum a_j^{(\tau_j)}(z)w^{\tau^j_s},\
\varphi^{(k)}(z,w)=\sum b_k^{(\mu_k)}(z)w^{\mu_k^{\ast}},\
g^{(s)}(z,w) =\sum d^{(j)}(z)w^{n_s^j},
\end{equation}
where the sums run over all indices such that
\begin{equation} \Label{inds1}
\tau_j+2\tau_s^j=s-1,\ \mu_k+2\mu_k^{\ast}=k,\ j+2n_s^j=s.
\end{equation}
Recall that $s=2s^*$.  Letting $w=0,\ \eta=2i\langle \-z,\xi
\rangle$, and collecting terms of degree $K>2$ in $\xi$ and degree
$P$ in $z$ in \eqref{Leq},  we obtain
\begin{multline}\Label{diffeq}
-2i\overline{a_j^{(K-P)}(\xi)}\overline{\eta^{P}}\delta_{j\ell}=
2i\sum_{k=2}^{s-2}\left\langle \varphi^{(k)}_{z_j}(z,0), \sum
\overline{b_{s-k}^{(\mu_{s-k})}(\xi)}\overline{\eta^{\mu^{\ast}_{s-k}}}
\right\rangle_{\tau}\\
 -4\sum_{k'=3}^{s-2}\delta_{j\ell}\bar \xi_j \left\langle
\varphi^{(k')}_w(z,0),\sum \overline{b_{s-k'}^{(\mu_{s-k'})}(\xi)}
\overline{\eta^{\mu^{\ast}_{s-k'}}} \right\rangle_{\tau},
\end{multline}
where the sums inside $\langle\cdot,\cdot\rangle_\tau$ run over the
indices
\begin{equation}\Label{ind2}
\mu_{s-k}+\mu_{s-k}^{\ast}=K,\quad k-1+\mu_{s-k}^{\ast}=P
\end{equation}
and
\begin{equation}\Label{ind2'}
\mu_{s-k'}+\mu_{s-k'}^{\ast}+1=K,\quad k'-2+\mu_{s-k'}^{\ast}=P
\end{equation}
Taking into account also the fact that
$\mu_{s-k}+2\mu_{s-k}^{\ast}=s-k$ and
$\mu_{s-k'}+2\mu_{s-k'}^{\ast}=s-k'$, we see that \eqref{ind2} and
\eqref{ind2'} have solutions only when
\begin{equation}
K+P+1=s.
\end{equation}
We shall let $K=s^*+p$  and $P=s^*-p-1$. We then get
\begin{equation}
\mu_{s-k}=k+2p,\quad \mu_{s-k}^{\ast}=s^*-p-k
\end{equation}
and
\begin{equation}
\mu_{s-k'}=k'+2p-2,\quad \mu_{s-k'}^{\ast}=s^*-p-k'+1.
\end{equation}
Since both $\mu_{s-k}^{\ast}$ and $\mu_{s-k'}^{\ast}$ must be
non-negative, we can rewrite \eqref{diffeq} as follows
\begin{multline}\Label{diffeq1}
-2i\overline{a_j^{(2p+1)}(\xi)}\overline{\eta^{s^*-p-1}}\delta_{j\ell}=
2i\sum_{k=2}^{s^*-p}\left\langle \varphi^{(k)}_{z_j}(z,0),
\overline{b_{s-k}^{(k+2p)}(\xi)}
\right\rangle_{\tau}\overline{\eta^{s^*-p-k}}\\
 -4\sum_{k'=3}^{s^*+1-p}\delta_{j\ell}\bar \xi_j \left\langle
\varphi^{(k')}_w(z,0),\overline{b_{s-k'}^{(k+2p-2)}(\xi)}
\right\rangle_{\tau}\overline{\eta^{s^*+1-p-k}},
\end{multline}
which can be rewritten as
\begin{multline}\Label{diffeq2}
-2i\overline{a_j^{(2p+1)}(\xi)}\overline{\eta^{s^*-p-1}}\delta_{j\ell}=\\
\sum_{k=2}^{s^*-p}\left(2i\left\langle \varphi^{(k)}_{z_j}(z,0),
\overline{b_{s-k}^{(k+2p)}(\xi)} \right\rangle_{\tau}
 -4\delta_{j\ell}\bar \xi_j \left\langle
\varphi^{(k+1)}_w(z,0),\overline{b_{s-k-1}^{(k+2p-1)}(\xi)}
\right\rangle_{\tau}\right)\overline{\eta^{s^*-p-k}},
\end{multline}
This equation is valid for all $p=0,1,\ldots, s^*-1$ and the sum on
the right hand side of \eqref{diffeq2} is void when $p=s^*-1$. When
$q\leq s^*-1$ we can use \eqref{phiids} to rewrite
\begin{equation}\Label{rewrite}
\begin{aligned}
\left\langle \varphi^{(q)}_{z_j}(z,0),
\overline{b_{s-q}^{(q+2p)}(\xi)} \right\rangle_{\tau}
&=\sum_{i=1}^{N-n}\delta_{i\tau} \varphi^{(q)}_{i,z_j}(z,0)
\overline{b_{i,s-q}^{(q+2p)}(\xi)}\\
&=\sum_{i=1}^\tau \varphi^{(q)}_{i,z_j}(z,0) \overline{c_i(\xi)},
\end{aligned}
\end{equation}
where
\begin{equation}
c_i(\xi):=\delta_{i\tau}b_{i,s-q}^{(q+2p)}(\xi)+\sum_{m=\tau+1}^{N-n}\delta_{m\tau}
a^{q}_{m i}b_{m,s-q}^{(q+2p)}(\xi).
\end{equation}
We can make a similar substitution in $\langle
\varphi^{(q)}_w(z,0),\overline{b_{s-q}^{(q+2p-2)}(\xi)}
\rangle_{\tau}$, again for $q\leq s^*-1$. Hence, if $2\leq p\leq
s^*-1$, then we conclude by Corollary \ref{CorI}, since $2\tau\leq
n-2$, that
\begin{equation}\Label{ajp}
a_j^{(2p+1)}(z)\equiv 0,\quad j=1,\ldots, n-1,\ p=2,\ldots s^*-1.
\end{equation}
When $p=1$, the equation \eqref{diffeq2} can be written as
\begin{multline}\Label{diffeqp1}
-2i\overline{a_j^{(3)}(\xi)}\overline{\eta^{s^*}}\delta_{j\ell}=\\
\sum_{k=2}^{s^*-2}\left(2i\left\langle \varphi^{(k)}_{z_j}(z,0),
\overline{b_{s-k}^{(k+2)}(\xi)} \right\rangle_{\tau}
 -4\delta_{j\ell}\bar \xi_j \left\langle
\varphi^{(k+1)}_w(z,0),\overline{b_{s-k-1}^{(k+1)}(\xi)}
\right\rangle_{\tau}\right)\overline{\eta^{s^*-1-k}}+\\
2i\left\langle \varphi^{(s^*-1)}_{z_j}(z,0),
\overline{b_{s^*+1}^{(s^*+1)}(\xi)} \right\rangle_{\tau}
 -4\delta_{j\ell}\bar \xi_j \left\langle
\varphi^{(s^*)}_w(z,0),\overline{b_{s^*}^{(s^*)}(\xi)}
\right\rangle_{\tau}
\end{multline}
We now turn to the equation \eqref{A4} in which we set $w=0$ and
$\eta=-2i\langle \bar z,\xi\rangle_{\ell}$. Collecting terms of
degree $s^*$ in $z$ and $s^*$ in $\xi$, we obtain
\begin{equation}
-\overline{d^{(0)}\eta^{s^{*}}}-2i\left \langle z,
\overline{a^{(1)}(\xi)}\right \rangle_{\ell}\,
\overline{\eta^{s^*-1}}=2i\sum_{k=2}^{s^*}\left \langle
\varphi^{(k)}(z,0),\overline{b^{(k)}_{s-k}(\xi)}\right \rangle_\tau
\overline{\eta^{s^*-k}}.
\end{equation}
This can be rewritten as
\begin{multline}
-\overline{d^{(0)}\eta^{s^{*}}}-2i\left \langle z,
\overline{a^{(1)}(\xi)}\right \rangle_{\ell}\,
\overline{\eta^{s^*-1}}-2i\sum_{k=2}^{s^*-1}\left \langle
b_k^{(k)}(z),\overline{b^{(k)}_{s-k}(\xi)}\right \rangle_\tau
\overline{\eta^{s^*-k}}= \\2i\left \langle
b^{(s^*)}_{s^*}(z),\overline{b^{(s^*)}_{s^*}(\xi)}\right
\rangle_\tau
\end{multline}
Since the left hand side is divisible by $\eta$ and $\tau<\ell$, it
follows from Lemma \ref{LemmaII} that
\begin{equation}\Label{bs*0}
\left \langle
b^{(s^*)}_{s^*}(z),\overline{b^{(s^*)}_{s^*}(\xi)}\right
\rangle_\tau\equiv 0,
\end{equation}
and there are constants $c^{s^*}_{jk}$ such that
\begin{equation}\Label{bs*}
b^{(s^*)}_{js^*}(z)=\sum_{k=1}^\tau
c^{s^*}_{jk}b^{(s^*)}_{ks^*}(z),\quad j=\tau+1,\ldots, N-n,
\end{equation}
where $b^{(s^*)}_{js^*}(z)$ is the $j$th component of the vector
valued function $b^{(s^*)}_{s^*}(z)$. By using Lemma \ref{LemmaI}
repeatedly (as in the proof of Corollary \ref{CorI}), we also
conclude that
\begin{equation}\Label{bkk}
\left \langle b_k^{(k)}(z),\overline{b^{(k)}_{s-k}(\xi)}\right
\rangle_\tau=0, \quad k=2,\ldots, s^*-1,
\end{equation}
and
\begin{equation}\Label{d00}
\overline{d^{(0)}\, \eta}=2i\left \langle z,
\overline{a^{(1)}(\xi)}\right \rangle_{\ell}, \end{equation}
Equation \eqref{d00} implies \eqref{rel0}. If we use \eqref{bs*} to
substitute for $b^{(s^*)}_{js^*}(z)$ in \eqref{diffeqp1} as above,
then we conclude, again by Corollary \ref{CorI}, that
\begin{equation}\Label{aj3}
a_j^{(3)}\equiv 0,\quad j=1,\ldots, n-1.
\end{equation}
Thus, the equations \eqref{ajp} and \eqref{aj3} together imply that
$f^{2s^*-1}(z,w)=a^{(1)}(z)w^{s^*-1}$.

To show that $g^{(2s^{\ast})}(z,w)=d^{(0)}w^{s^*}$, we go back to
\eqref{A4} in which we again set $w=0$ and $\eta=-2i\langle \bar
z,\xi\rangle_{\ell}$. Note that the degree of
$\overline{f^{2s^*-1}(\xi,\eta)}$ in $\xi$ is $s^*$ by what we have
already proved. Thus, if we collect terms of degree $K=s^*+p$ in
$\xi$ with $p\geq 1$, then we obtain
\begin{equation}\Label{homoK}
-\overline{d^{(2p)}(\xi)\eta^{s^*-p}}=2i\sum_{k=2}^{s-2}\left\langle
\varphi^{(k)}(z,0),\sum \overline{b_{s-k}^{(\mu_{s-k})}(\xi)}
\overline{\eta^{\mu^*_{s-k}}} \right\rangle_{\tau},
\end{equation}
where the sum inside $\langle \cdot,\cdot \rangle_\tau$ runs over
the indices
$$
\mu_{s-k}+\mu^*_{s-k}=s^*+p,\quad \mu_{s-k}+2\mu^*_{s-k}=s-k.
$$
 As
above, equation \eqref{homoK} can be rewritten as
\begin{equation}
-\overline{d^{(2p)}(\xi)\eta^{s^*-p}}=2i\sum_{k=2}^{s^*-p}\left\langle
\varphi^{(k)}(z,0),\overline{b_{s-k}^{(k+2p)}(\xi)}
\right\rangle_{\tau}\overline{\eta^{s^*-p-k}},
\end{equation}
As above, letting $p=1,\ldots, s^*$ (understanding the sum on the
right to be void when $p\geq s^*-2$) and substituting for
$\varphi^{(k)}(z,0)$ using \eqref{phiids}, we conclude that
\begin{equation}
d^{(2p)}(\xi)\equiv 0,\quad p=1,\ldots, s^*.
\end{equation}
This completes the proof of Claim \ref{mainclaim}.
\end{proof}

\begin{rem}{\rm
For future reference, we note that we may combine equations
\eqref{bs*0} and \eqref{bkk} to conclude
\begin{equation}\Label{bkk1}
\left \langle b_k^{(k)}(z),\overline{b^{(k)}_{s-k}(\xi)}\right
\rangle_\tau=0, \quad k=2,\ldots, s^*.
\end{equation}
By complex conjugating and switching the roles of $z$ and $\xi$ in
\eqref{bkk1}, we also obtain
\begin{equation}\Label{bkk2}
\left \langle b_k^{(s-k)}(z),\overline{b^{(s-k)}_{s-k}(\xi)}\right
\rangle_\tau=0, \quad k=s^*,\ldots, s-2.
\end{equation}
}
\end{rem}

 To complete the proof of \eqref{mainlem}, we need to employ the
{\it moving point} trick first  introduced in \cite{Hu1}, and also
crucially used in the later works \cite{Hu2} and \cite{BH}. We first
recall some notation and definitions from \cite{Hu2} and \cite{BH}.

Let $F_p=\tau^F_p \circ F \circ \sigma_p^0$ be as in (3.1) in \cite{
BH} and let
$F_p^{\ast\ast}=(f_p^{\ast\ast},\varphi_p^{\ast\ast},g_p^{\ast\ast})$
be its second normalization at the base point  $p\in M$ (see p. 390
in \cite{BH}). In particular, we have
\begin{align}\Label{bbb}
\left(f_p^{\ast\ast}\right)_j=\frac{(f_p^{\ast})_j-a_j(p)g_p^{\ast}}
{1+2i\left\langle \tilde{f}_p^{\ast},\overline{a(p)}
\right\rangle_{\ell,\ell',n}+(r(p)-i|a(p)|^2_{\ell,\ell',n})g_p^{\ast}},\quad
j=1,\ldots,n-1.
\end{align}
Here, $ {a(p)}=T(\tilde f_p^{\ast})(0)$, which can be written
$$
a_j(p)=\frac{1}{\lambda(p)}\left\langle
T(\tilde{f}),\overline{L_j(\tilde{f})} \right\rangle_{\ell,\ell',n}
\bigg|_p\ \mbox{for}\quad j\leq n-1.
$$
and
$$
a_j(p)=\frac{1}{\lambda(p)}\left\langle
T(\tilde{f}),\overline{D_j(p)}\right\rangle_{\ell,\ell',n}\bigg|_p \
\mbox{for}\quad j \geq n,
$$
where
$$
L_j=\frac{\partial}{\partial z_j}+2i\delta_{jl}\bar
z_j\frac{\partial}{\partial w},\quad j=1,\ldots,n-1,\quad
T=\frac{\partial}{\partial w},
$$
the vectors $D_j(p)$ and the scalar $\lambda(p)$ depend on first
order derivatives of the mapping $F$ at $p$. The reader is referred
to Sections 2 and 3 in \cite{BH}, and also \cite{Hu2}, for details
and notation.

We state here some facts that can be found in \cite{BH}. As
functions of $p\in \mathbb{H}^n_\ell$, $deg_{wt}(a_j(p))\geq 2$ for
$j\leq n-1$ and $deg_{wt}a_j(p)\geq 1$ for $j\geq n$. To emphasize
that the weighted degree is as a function of $p$, we shall write
$a_j(p)=O_{wt,p}(1)$, $j\geq n$, and so on in what follows.
 Also, we have

\begin{align}\Label {A12}
\begin{split}
&|a(p)|^2_{\ell,\ell',n}=\frac{1}{\lambda(p)}|T(\tilde{f})|^2_{\ell,\ell',n}=O_{wt,p}(2),\
(r+i|a(p)|^2_{\ell,\ell',n})=O_{wt,p}(2),\  \l(p)=1+O_{wt,p}(2),\\
&g^{\ast}_{p}=\frac{1}{\lambda(p)}\left(g-2i\left\langle
\tilde{f},\overline{\tilde{f}(p)}\right\rangle_{\ell,\ell',n}\right)
\circ
\sigma_0^p,\\
&\tilde{f}^{\ast}_{p}=(\tilde{f}-\tilde{f}(p)) \circ \sigma_0^p
\cdot
\tilde{A}^{-1},\\
&(f^{\ast}_p)_j=\d_{j,\ell}\left\langle
\frac{1}{\lambda}(\tilde{f}-\tilde{f}(p))\circ
\sigma_0^p,\overline{L_j(\tilde{f})} \right\rangle_{\ell,\ell',n}.
\end{split}
\end{align}
We observe further that $F^{**}_p$ is a normalized map sending the
origin to the origin and an open piece of $\mathbb H^n_\ell$ into
$\mathbb H^N_{\ell,\ell'}$. Hence, the conclusion of Claim
\ref{mainclaim} holds for the components $f^{**}_p$ and $g^{**}_p$
for every $p\in\mathbb H^n_\ell$ near 0. We shall prove the
following lemma, which then completes the proof of \eqref{mainlem}.

\begin{lem}\Label{mainlem2} Suppose that the hypotheses in Proposition 3.1
hold. Assume further that for any $p\in M$, we have
\begin{equation} \label{0000}
f^{\ast\ast}_p(z,w)=z+a_p^{(1)}(z)w^{s^{\ast}-1}+O_{wt}(s),\ \ \
g^{\ast\ast}_p(z,w)=w+d^{(0)}_pw^{s^{\ast}}+O_{wt}(s+1),
\end{equation}
where $s=2s^*$. Then, $a_p^{(1)}(z)\equiv 0$,  $d^{(0)}_p=0$, and hence $$f(z,w)=z+O_{wt}(s),\ g(z,w)=w+O_{wt}(s+1).$$
 \end{lem}
We remark that by the hypothesis in Proposition \ref{mainprop} and
what we did above, the assumption in (\ref{0000}) always holds.

\begin{proof}[Proof of Lemma $\ref{mainlem2}$]
Let us identify the coefficients, as functions of $p\in \mathbb
H^n_\ell$, in front of $w^{s^*-1}$ in the Taylor expansion on both
sides of \eqref{bbb} for a fixed $1\leq j\leq n-1$. We note, in view
of the assumption in the lemma, that the coefficient on the left is
0. To find the coefficient on the right, we first consider the
coefficient of $(f^*_p)_j$, which in view of \eqref{A12} equals
\begin{equation}
\frac{1}{\lambda(p)}\left\langle
T^{s^{\ast}-1}\tilde{f}(p),\overline{L_j\tilde{f}(p)}
\right\rangle_{\ell,\ell',n}.
\end{equation}
If we use the following expansion of $\tilde f=(f,\varphi)$,
\begin{equation}
f_j(z,w)=z+a^{(1)}_j(z)w^{s^{\ast}-1}+O_{wt}(s),\
\varphi(z,w)=\varphi^{(2)}(z)+\varphi^{(3)}(z,w)+...+\varphi^{(s-2)}(z,w)+...
\end{equation}
with
\begin{equation}
\varphi^{(s-2)}(z,w)=b^{(0)}_{s-2}w^{s-2}+\sum_{\mu^*_{s-2}=0}^{s-3}b_{s-2}
^{(s-2-2\mu^*_{s-2})}(z)
w^{\mu^*_{s-2}},
\end{equation}
then, with $p=(z_p,w_p)$, we obtain
\begin{equation}\Label{leftcoeff}
\frac{1}{\lambda(p)}\left\langle
T^{s^{\ast}-1}\tilde{f}(p),\overline{L_j\tilde{f}(p)}
\right\rangle_{\ell,\ell',n}=(s^*-1)!\,\left(
\delta_{j\ell}\,a_j^{(1)}(z_p)+\left\langle
b_{s-2}^{(0)},\overline{\varphi_{z_j}^{(2)}(z_p)}
\right\rangle_{\tau}\right )+O_{wt,p}(2).
\end{equation}
Since $a_j(p)=O_{wt,p}(2)$, for $j=1,\ldots, n-1$, and the
denominator of \eqref{bbb} is $1+O_{wt,p}(1)$, we conclude that the
coefficient of $w^{s^*-1}$ on the right hand side of \eqref{bbb}
equals \eqref{leftcoeff} modulo $O_{wt,p}(2)$. The conclusion is
that
\begin{equation}
\delta_{j\ell}\,a_j^{(1)}(z_p)+\left\langle
b_{s-2}^{(0)},\overline{\varphi_{z_j}^{(2)}(z_p)}\right\rangle_{\tau}=0.
\end{equation}
Since this holds for all $j=1,\ldots, n-1$ and all $p\in \mathbb
H^n_\ell$, and hence for all $z_p$ in a neighborhood of $0$, we
conclude that
$$
a_j^{(1)}(z)\equiv 0,\quad j=1,\ldots, n-1.
$$
This also implies that $g(z,w)=w+O_{wt}(s+1)$ in view of
\eqref{rel0}. This completes the proof of Lemma \ref{mainlem2}.
\end{proof}

To complete the induction step in the proof of Proposition
\ref{mainprop}, we must show that
$$
\left\langle
\varphi^{(k)}(z,w),\overline{\varphi^{(s-k)}(\xi,\eta)}\right\rangle_\tau=0,\quad
k=2,\ldots, s-2,
$$
for $s=2s^*$. Using the expansion of $\varphi^{(k)}(z,w)$ in
\eqref{expans}, we observe that it suffices to show that
\begin{equation}\Label{goal}
\left\langle
b_k^{(\mu_k)}(z),\overline{b_{s-k}^{(\mu^*_{s-k})}(\xi)}\right\rangle_\tau=0
\end{equation}
for all indices such that
\begin{equation}\Label{index}
k=2,\ldots, s-2,\quad \mu_k+2\mu^*_{k}=k,\quad
\mu_{s-k}+2\mu^*_{s-k}=s-k.
\end{equation}
For this purpose, let us consider equation \eqref{A4} with $w=\bar
\eta+2i\langle z,\bar \xi\rangle_\ell$ and collect terms of degree
$p$ in $\eta$, $q$ in $z$, and $r$ in $\xi$. Since the left hand
side of \eqref{A4} vanishes by \eqref{mainlem}, we obtain
\begin{equation}\Label{rhs}
\begin{aligned}
\sum_{k=2}^{s-2} \sum
\sum_{j=0}^{\mu^*_k}(2i)^{\mu^*_k-j}\binom{\mu^*_k}{j}\bar\eta^{\mu^*_{s-k}+j}\left\langle
z,\bar \xi\right\rangle_\ell^{\mu^*_k-j} \left\langle
b_k^{(\mu_k)}(z),\overline{b_{s-k}^{(\mu_{s-k})}(\xi)}\right\rangle_\tau=0,
\end{aligned}
\end{equation}
where the notation introduced in \eqref{expans} has been used and as
above $s=2s^*$. The middle sum in \eqref{rhs} ranges over those
indices $(\mu_k,\mu^*_k,\mu_{s-k},\mu^*_{s-k})$ that satisfy the
equations
\begin{equation}
\begin{aligned}
\mu_k+\mu^*_k-j &=q\\
\mu_k+2\mu^*_k &=k\\
\mu^*_k+\mu_{s-k}-j &=r\\
\mu_{s-k}+2\mu^*_{s-k} &=s-k\\
\mu^*_{s-k}+j &=p.
\end{aligned}
\end{equation}
For fixed $2\leq k\leq s-2$, we may view the two last sums in
\eqref{rhs} as ranging over those $\mu_k$, $\mu^*_k$, $\mu_{s-k}$,
$\mu^*_{s-k}$ and $j$ that satisfy the following system of
equations:
\begin{equation}\Label{musystem}
\begin{pmatrix}
1 & 1& 0& 0& -1\\
1 & 2& 0& 0&0 &\\
0& 1& 1& 0& -1\\
0& 0& 1& 2& 0\\
0& 0&0 &1& 1
\end{pmatrix}
\begin{pmatrix}
\mu_k\\ \mu^*_k\\
\mu_{s-k}\\
\mu^*_{s-k}\\j
\end{pmatrix}
=
\begin{pmatrix}
q\\
k\\
r\\
s-k\\
p
\end{pmatrix}
\end{equation}
subject to the constraints
\begin{equation}\Label{constraints}
\mu_k\geq 0,\ \mu^*_k\geq 0,\ \mu_{s-k}\geq 0,\ \mu^*_{s-k}\geq 0,\
0\leq j\leq \mu^*_k.
\end{equation}
A straighforward row reduction shows that \eqref{musystem} is
equivalent to the system
\begin{equation}\Label{musystem2}
\begin{pmatrix}
1 & 1& 0& 0& -1\\
0 & 1& 0& 0&1 &\\
0& 0& 1& 0& -2\\
0& 0& 0& 2& 2\\
0& 0&0 &1& 1
\end{pmatrix}
\begin{pmatrix}
\mu_k\\ \mu^*_k\\
\mu_{s-k}\\
\mu^*_{s-k}\\j
\end{pmatrix}
=
\begin{pmatrix}
q\\
k-q\\
r+q-k\\
s-r-q\\
p
\end{pmatrix}
\end{equation}
We conclude that the system is solvable when
\begin{equation}
2p+q+r=s,
\end{equation}
in which case we may solve for $\mu_k$, $\mu^*_k$, $\mu_{s-k}$, and
$\mu^*_{s-k}$ in terms of $j$ to obtain
\begin{equation}\Label{soln}
\begin{aligned}
\mu_k &=2q+2j-k\\
\mu^*_k &=k-q-j\\
\mu_{s-k} &= r+q+2j-k\\
\mu^*_{s-k} &= p-j.
\end{aligned}
\end{equation}
Let us begin by choosing $p=0$ and, for integral $m$ with
$-(s^*-2)\leq m\leq s^*-2$, choose $q=s^*+m$ and $r=s^*-m$. The
constraints in \eqref{constraints} imply that $j=0$ and $s^*+m\leq
k\leq s+2m$. Note that in this induction step of the proof of
Proposition \ref{mainprop}, only $b^{(\mu_k)}(z)$ with $2\leq k\leq
s-2$ are involved. For simplicity of notation, with $s=2s^*$ fixed,
let us introduce $b_k^{(\mu_k)}(z)$ for any integer $k$ by defining
\begin{equation}\Label{bknot}
b_k^{(\mu_k)}(z)\equiv 0,\ \text{if $k\leq 1$ or $k\geq s-1$}.
\end{equation}
We then conclude from \eqref{A4} and the above that
\begin{equation}\Label{p0qrm}
\begin{aligned}
\sum_{k=s^*+m}^{s+2m} (2i)^{k-s^*-m}\left\langle z,\bar
\xi\right\rangle_\ell^{k-s^*-m} \left\langle
b_k^{(s+2m-k)}(z),\overline{b_{s-k}^{(s-k)}(\xi)}\right\rangle_\tau=0.
\end{aligned}
\end{equation}
Note that with the convention \eqref{bknot}, the sum really ranges
over $k$ from $\max (2,s^*+m)$ and $\min (s-2,s+2m)$. If $m\geq 1$
or $m< -s^*/2$, then no term of the form $b_{s^*}^{(\mu_{s^*})}(z)$
(or its complex conjugate) appears in the sum in \eqref{p0qrm} and
it follows from Corollary \ref{CorI} and \eqref{phiids} (as in the
proof of Claim \ref{mainclaim} above) that
\begin{equation}
\left\langle
b_k^{(s+2m-k)}(z),\overline{b_{s-k}^{(s-k)}(\xi)}\right\rangle_\tau=0,
\quad k=s^*+m,\ldots, s+2m.
\end{equation} With $-s^*/2\leq m\leq 0$, we can rewrite \eqref{p0qrm}
as follows
\begin{multline}\Label{p0qr-1}
\sum_{k=s^*+m}^{s^*-1} (2i)^{k-s^*-m}\left\langle z,\bar
\xi\right\rangle_\ell^{k-s^*-m} \left\langle
b_k^{(s+2m-k)}(z),\overline{b_{s-k}^{(s-k)}(\xi)}\right\rangle_\tau
+\\
(2i)^{-m}\left\langle z,\bar \xi\right\rangle_\ell^{-m}\,
\left\langle
b_{s^*}^{(s^*+2m)}(z),\overline{b_{s^*}^{(s^*)}(\xi)}\right\rangle_\tau\\
+ \sum_{k=s^*+1}^{s+2m} (2i)^{k-s^*-m}\left\langle z,\bar
\xi\right\rangle_\ell^{k-s^*-m} \left\langle
b_k^{(s+2m-k)}(z),\overline{b_{s-k}^{(s-k)}(\xi)}\right\rangle_\tau=0,
\end{multline}
where we understand any sum over an index set such that the upper
limit is strictly less than the lower limit to not be present. By
using \eqref{bs*} to rewrite (as in \eqref{rewrite})
$$ \left\langle
b_{s^*}^{(s^*+2m)}(z),\overline{b_{s^*}^{(s^*)}(\xi)}\right\rangle_\tau
$$
appearing in the middle term in \eqref{p0qr-1}, and \eqref{phiids}
to rewrite the remaining terms, we conclude from Corollary
\ref{CorI} that also in this case
\begin{equation}
\left\langle
b_k^{(s-2-k)}(z),\overline{b_{s-k}^{(s-k)}(\xi)}\right\rangle_\tau=0,
\quad k=s^*+m,\ldots s+2m.
\end{equation}
We can summarize the above, using also the convention \eqref{bknot},
as follows
\begin{equation}\Label{bkkind}
\left\langle
b_k^{(s+2m-k)}(z),\overline{b_{s-k}^{(s-k)}(\xi)}\right\rangle_\tau=0,
\quad s^*+m\leq k\leq s+2m,\ \ -s^*\leq m \leq s^*.
\end{equation}
Now, for general $0\leq p\leq s^*$, $q=s^*+m-p$ and $r=s^*-m-p$, the
equations \eqref{soln} can be rewritten
\begin{equation}\Label{soln2}
\begin{aligned}
\mu_k &=s+2m-2p+2j-k\\
\mu^*_k &=k-s^*-m+p-j\\
\mu_{s-k} &= s-2p+2j-k\\
\mu^*_{s-k} &= p-j.
\end{aligned}
\end{equation}
The constraints in \eqref{constraints} amount to $0\leq j\leq p$ and
\begin{equation}\Label{pconstr2} s^*+m-p+2j\leq k\leq
\min(s-2p+2j,s+2m-2p+2j).
\end{equation}
We can rewrite  \eqref{rhs} as follows
\begin{multline}\Label{rhspm}
\sum_{j=0}^p\sum_{k=s^*+m-p+2j}^{s(+2m)-2p+2j}
(2i)^{k-s^*-m+p-2j}\binom{k-s^*-m+p-j}{j}\bar\eta^{p}\times
\\
\left\langle z,\bar \xi\right\rangle_\ell^{k-s^*-m+p-2j}
\left\langle
b_k^{(s+2m-2p+2j-k)}(z),\overline{b_{s-k}^{(s-2p+2j-k)}(\xi)}\right\rangle_\tau=0,
\end{multline}
where the parenthetical $(+2m)$ in the upper limit in the sum over
$k$ is only present when $m<0$.

\begin{claim}\Label{mainclaim2}
We have, for every $0\leq p\leq s^*$,
\begin{equation}\Label{bkkf}
\left\langle
b_k^{(s+2m-2p-k)}(z),\overline{b_{s-k}^{(s-2p-k)}(\xi)}\right\rangle_\tau=0,
\end{equation}\Label{mless0}
for all \begin{equation} s^*+m-p\leq k\leq s+2m-2p,\ \text{\rm when
$-(s^*-p)\leq m<0$}
\end{equation}
and
\begin{equation}\Label{mgr0} s^*+m-p\leq k\leq s^*-p,\ \text{\rm when
$0\leq m\leq s-2p$},
\end{equation}
where we understand $b_k^{(\mu_k)}(z)\equiv0$ if $k\leq 1$ or
$k\geq s-1$.
 In addition, there
are constants $c^{s^*-2l}_{jk}$ such that
\begin{equation}\Label{bs*l}
b^{(s^*-2l)}_{js^*}(z)=\sum_{k=1}^\tau
c^{s^*-2l}_{jk}b^{(s^*-2l)}_{ks^*}(z),\quad j=\tau+1,\ldots, N-n,\ \
l=0,1,\ldots,\min(p,[s^*/2]),
\end{equation}
where $b^{(s^*-2l)}_{js^*}(z)$ is the $j$th component of the vector
valued function $b^{(s^*-2l)}_{s^*}(z)$ and $[s^*/2]$ denotes the
largest integer $\leq s^*/2$.
\end{claim}

\begin{proof}[Proof of Claim $\ref{mainclaim2}$] We shall prove Claim \ref{mainclaim2}
by induction on $p$. We observe that for $p=0$, the equations in
\eqref{bkkf} reduce to \eqref{bkkind}. (Note that
 \eqref{bkkf} is automatic for $k\leq 1$ and $k\geq s-1$ by the convention introduced
 above.) Also, \eqref{bs*l} is just
\eqref{bs*}. Let $p^*<s^*$ and assume that \eqref{bkkf},
 and \eqref{bs*l} hold for all $p\leq p^*$. Consider
the term
\begin{equation}
\left\langle
b_k^{(s+2m-2p+2j-k)}(z),\overline{b_{s-k}^{(s-2p+2j-k)}(\xi)}\right\rangle_\tau
\end{equation}
in the sum \eqref{rhspm} with $p=p^*+1$ and $m=m^*$, i.e.\
\begin{equation}\Label{jgeq1}
\left\langle
b_k^{(s+2m^*-2p^*-2+2j-k)}(z),\overline{b_{s-k}^{(s-2p^*-2+2j-k)}(\xi)}\right\rangle_\tau,
\end{equation}
where $0\leq j\leq p^*+1$ and
\begin{equation}\Label{kpm*}
s^*+m^*-p^*-1+2j\leq k \leq s(+2m^*)-2p^* -2+2j
\end{equation}
where the parenthetical term $(+2m^*)$ is only present when $m^*<0$.
 By the induction hypothesis, this term vanishes if there are
$-(s^*-p)\leq m\leq s^*-p$ and $0\leq p\leq p^*$ such that
\begin{equation}\Label{ineq}
\begin{aligned}
& 2m-2p=2m^*-2p^*-2+2j\\
& 2p=2p^*+2-2j
\end{aligned}
\end{equation}
satisfying \begin{equation}\Label{kpm}
 s^*+m-p\leq k \leq s(+2m) -2p,
 \end{equation}
 where the parenthetical term $(+2m)$ is only present when $m<0$.
 (In what follows, we shall continue to use this convention.)
Solving for $m$ and $p$ in \eqref{ineq}, we obtain
\begin{equation}
m=m^*,\quad p=p^*+1-j.
\end{equation}
The constraint $p\leq p^*$ is satisfied if $j\geq 1$. We note that
$$s^*+m-p=s^*+m^*-p^*-1+j\leq s^*+m^*-p^*-1+2j\leq k.$$ We also have
$s(+2m)-2p=s(+2m^*)-2p^*-2+2j\geq k$. We conclude that the
constraint \eqref{kpm} is satisfied. Hence, the terms \eqref{jgeq1}
that appear in \eqref{rhspm} (with $p=p^*+1$ and $m^*$) vanish for
all $j\geq 1$. Consequently, it follows from \eqref{rhspm}, with
$p=p^*+1$ and $m=m^*$, that
\begin{equation}\Label{rhspm2}
\sum_{k=s^*+m-p^*-1}^{s(+2m)-2p^*-2} (2i)^{k-s^*-m+p^*+1}
\left\langle z,\bar \xi\right\rangle_\ell^{k-s^*-m+p^*+1}
\left\langle
b_k^{(s+2m-2p^*-2-k)}(z),\overline{b_{s-k}^{(s-2p^*-2-k)}(\xi)}\right\rangle_\tau=0.
\end{equation}
Assume first that $p^*\leq s^*/2-1$. Then if $m\geq p^*+2$ or $m\le
-s^*/2+p^*+1/2$, no term involving $b_{s^*}^{(\mu_{s^*})}(z)$
appears in the sum and equation \eqref{rhspm2} implies, by using
Corollary \ref{CorI} and \eqref{phiids} as above, that
\begin{equation}\Label{temp}
\left\langle
b_k^{(s+2m-2p^*-2-k)}(z),\overline{b_{s-k}^{(s-2p^*-2-k)}(\xi)}\right\rangle_\tau=0,\quad
k=s^*+m-p^*-1,\ldots, s(+2m)-2p^*-2.
\end{equation}
 When $-s^*/2+p^*+1/2< m\leq p^*+1$, we can rewrite
\eqref{rhspm2} as follows
\begin{multline}\Label{rhspm3}
 \sum_{k=s^*+m-p^*-1}^{s^*-1} (2i)^{k-s^*-m+p^*+1} \left\langle
z,\bar \xi\right\rangle_\ell^{k-s^*-m+p^*+1} \left\langle
b_k^{(s+2m-2p^*-2-k)}(z),\overline{b_{s-k}^{(s-2p^*-2-k)}(\xi)}\right\rangle_\tau+\\
(2i)^{-m+p^*+1}\left\langle z,\bar \xi\right\rangle_\ell^{-m+p^*+1}
\left\langle
b_{s^*}^{(s^*+2m-2p^*-2)}(z),\overline{b_{s^*}^{(s^*-2p^*-2)}(\xi)}
\right\rangle_\tau+\\
\sum_{k=s^*+1}^{s(+2m)-2p^*-2} (2i)^{k-s^*-m+p^*+1} \left\langle
z,\bar \xi\right\rangle_\ell^{k-s^*-m+p^*+1} \left\langle
b_k^{(s+2m-2p^*-2-k)}(z),\overline{b_{s-k}^{(s-2p^*-2-k)}(\xi)}\right\rangle_\tau=0,
\end{multline}
where we understand any sum over an index set such that the upper
limit is strictly less than the lower limit to not be present. Let
us first consider the case $1\leq m\leq p^*+1$. Note that in this
case $s^*+2m-2p^*-2$ can be written as $s^*-2l$ with $l=p^*+1-m$ and
then $l\leq p^*$, and hence we can use the induction hypothesis
\eqref{bs*l} to rewrite the middle term
$$
\left\langle
b_{s^*}^{(s^*+2m-2p^*-2)}(z),\overline{b_{s^*}^{(s^*-2p^*-2)}(\xi)}
\right\rangle_\tau
$$
as a sum involving only $\tau$ terms (as in \eqref{rewrite}). All
other terms can be similarly rewritten using \eqref{phiids}. We
conclude by Corollary \ref{CorI} that \eqref{temp}, which was shown
above to hold for $m\geq p^*+2$ and $m<-s^*/2+p^*+1$, holds also for
$1\leq m\leq p^*+1$.

Next, we consider $m=0$. In this case, we first use Lemma
\ref{LemmaI} repeatedly (as in the proof of Corollary \ref{CorI}) to
deduce that each term in the first sum in \eqref{rhspm3} vanishes,
i.e.\
\begin{equation}
 \left\langle
b_k^{(s-2p^*-2-k)}(z),\overline{b_{s-k}^{(s-2p^*-2-k)}(\xi)}\right\rangle_\tau=0,\quad
k=s^*-p^*-1,\ldots, s^*-1.
\end{equation}
By canceling a factor $2i\langle z,\bar \xi\rangle^{p^*+1}$, we are
left with
\begin{multline}\Label{rhspm4}
\left\langle
b_{s^*}^{(s^*-2p^*-2)}(z),\overline{b_{s^*}^{(s^*-2p^*-2)}(\xi)}
\right\rangle_\tau+\\
\sum_{k=s^*+1}^{s-2p^*-2} (2i)^{k-s^*} \left\langle z,\bar
\xi\right\rangle_\ell^{k-s^*} \left\langle
b_k^{(s-2p^*-2-k)}(z),\overline{b_{s-k}^{(s-2p^*-2-k)}(\xi)}\right\rangle_\tau=0,
\end{multline}
It follows from Lemma \ref{LemmaII} that
$$\left\langle
b_{s^*}^{(s^*-2p^*-2)}(z),\overline{b_{s^*}^{(s^*-2p^*-2)}(\xi)}
\right\rangle_\tau=0
$$
and
\begin{equation}\Label{bs*p*}
b^{(s^*-2p^*-2)}_{js^*}(z)=\sum_{k=1}^\tau
c^{s^*-2p^*-2}_{jk}b^{(s^*-2p^*-2)}_{ks^*}(z),\quad j=\tau+1,\ldots,
N-n
\end{equation}
for some constants $c^{s^*-2p^*-2}_{jk}$. Note that \eqref{bs*p*}
and the induction hypothesis yield \eqref{bs*l} with $p=p^*+1$. Now,
applying Corollary \ref{CorI} to the remaining equation, we conclude
that \eqref{temp} holds for $m=0$.

Finally, we consider the remaining cases $-s^*/2+p^*+1\leq m\leq
-1$. We can rewrite the terms in the first and last sum using
\eqref{phiids} and the middle term using \eqref{bs*p*}. We conclude,
by Corollary \ref{CorI}, that \eqref{temp} holds for these $m$ as
well. Consequently, we have proved \eqref{bkkf}, for the ranges of
$k$ and $m$ given by \eqref{mless0} and \eqref{mgr0} and $l=p^*+1$,
when $p^*\leq s^*/2-1$.

We now assume $p^*>s^*/2-1$. In this case the upper limit in the sum
over $k$ in \eqref{rhspm2} satisfies $s(+2m)-2p^*-2<s^*$ for all
$m$, since the term $(+2m)$ is only present when $m<0$. Hence, no
term $b^{(\mu_{s^*})}_{s^*}(z)$ will appear in the equation. The
conclusion \eqref{bkkf}, for the ranges of $k$ and $m$ given by
\eqref{mless0} and \eqref{mgr0} and $p=p^*+1$, then follows also in
this case. Note that when $p^*>s^*/2-1$, then $\min
(p^*+1,[s^*/2])=[s^*/2]\leq p^*$ and, hence, \eqref{bs*l} trivially
holds also for $p=p^*+1$ by the induction hypothesis. Claim
\ref{mainclaim2} now follows by induction.
\end{proof}

To prove \eqref{goal} for all indices such that \eqref{index} holds,
it suffices in view of Claim \ref{mainclaim2} to show that for any
integer $2\leq k\leq s-2$ and any integers  $0\leq \mu_k\leq k$,
$0\leq \mu_{s-k}\leq s-k$ of the same parity as $k$, we can find
integers $m$ and $p$ satisfying
\begin{equation}\Label{finaleq}
s+2m-2p-k=\mu_k,\quad s-2p-k=\mu_{s-k}
\end{equation}
such that the constraints
\begin{equation}\Label{constr2}
0\leq p\leq s^*,\quad -(s^*-p)\leq m\leq s^*-p,\quad s^*+m-p\leq
k\leq s(+2m)-2p
\end{equation}
hold, where again the parenthetical term $(+2m)$ is only present
when $m<0$. We shall treat here the case of even $k=2l$, with $1\leq
l\leq s^*-1$. The case of odd $k$ is similar and left to the reader.
We can write $\mu_k=2x$ and $\mu_{s-k}=2y$, where $0\leq x\leq l$
and $0\leq y\leq s^*-l$. The unique solution to \eqref{finaleq} is
then
\begin{equation}
p=s^*-y-l,\quad m=x-y.
\end{equation}
Clearly, we have $0\leq p\leq s^*$. We find that $s^*-p=y+l$. It
follows that $m\leq x\leq l\leq y+l=s^*-p$ and $m\geq -y\geq
-(y+l)=-(s^*-p)$, and hence $m$ satisfies the middle constraints in
\eqref{constr2}. Moreover, we have $s^*+m-p=x+l\leq 2l=k$. When
$m\geq 0$, the upper limit for $k$ in \eqref{constr2} equals
$s(+2m)-2p=s-2p=2(s^*-p)=2y+2l\geq 2l=k$. When $m<0$, the upper
limit for $k$ equals $s(+2m)-2p=2(s^*-p+m)=2x+2l\geq 2l=k$. We
conclude that all constraint in \eqref{constr2} are satisfied. Thus,
we have proved \eqref{goal} for all indices such that the
constraints in \eqref{index} hold. This completes the induction step
for $s=2s^*$ and, hence, the proof of Proposition \ref{mainprop} is
complete.
\end{proof}
\begin{proof}[Proof of Theorem $\ref{main}$ {\rm(a)}] As explained in
the beginning of this section, we may assume that the mapping $F$
satisfies the normalization \eqref{norm}. By induction and
Proposition 3.1, we conclude that
$$f(z,w)\equiv z, \ \ g(z,w)\equiv w,\ \
\left\langle\varphi(z,w),\overline{\varphi(\xi,\eta)}
\right\rangle_\tau\equiv 0,$$
where $\tau=\ell' - \ell$ as above. By a lemma of D'Angelo, there is a
constant $(N-n-\tau)\times(N-n-\tau)$ matrix ${\mathcal U}$ such
that
$$
{\mathcal U}\overline{{\mathcal U}^t}=I_{(N-n-\tau)\times(N-n-\tau)}
$$
and
$$(\varphi_{\tau+1},\cdots,\varphi_{N-n})\mathcal U=(\varphi_1,\cdots,\varphi_{\tau},0,
\ldots,0).
$$
If we let $\gamma$ be the automorphism of $\mathbb H^N_{\ell'}$ given
by
$$
\gamma(z,w):=(z',z'',z'''\mathcal U,w),
$$
where $z=(z',z'',z''')$ with
$$
z'=(z_1,\ldots,z_{n-1}),\quad z''=(z_n,\ldots,z_{n+\tau-1}),\quad
z'''=(z_{n+\tau},\ldots,z_{N-1}).
$$
then $\gamma\circ F$ satisfies the conclusion of Theorem \ref{main}
(a). The proof of Theorem \ref{main} (a) is complete.
\end{proof}

\begin{proof}[Proof of Theorem $\ref{main}$ {\rm (b)}]

We now assume that  $\frac{\p g}{\p w}(0)=\lambda < 0$.
 By counting
the number of  negative and positive eigenvalues in [(2.6), BH], we
similarly see that $(n-1-\ell)\le \ell'$ and $N-1-\ell'\ge \ell.$

Assume that $\ell'<n-1$. Define $\ell^*=n-1-\ell$ and
$\tau^*=\ell'-(n-1-\ell)$. Then we have
 $\tau^*<\ell\le \frac{n-1}{2}$. Notice that Lemma $\ref{LemmaII}$ still holds
when $\ell, \tau$ are replaced by $\ell^*$ and $\tau^*$,
respectively. Indeed, to see this,  we need only to observe that
$A(z,\-{z})|z|_{\ell^*}=-A(z,\-{z}) (-|z_{\ell^*+1}|^2-\cdots
-|z_{n-1}|^2+|z_{1}|^2+\cdots +|z_{\ell^*}|^2)$ with
$\frac{n-1}{2}\ge n-1-\ell^*>\tau^*.$ Certainly Lemma $\ref{LemmaI}$
holds when $\ell$ is replaced by $\ell^*$.

Let
$\sigma^*(z,w)=(z_{\ell^*+1},\cdots,z_{n-1},z_1,\cdots,z_{\ell^*},-w)$
and consider $F^*=F\circ \sigma^*$. Then $F^*$ maps  a small piece
of ${\mathbb H}^{n}_{\ell^*}$ near the origin into ${\mathbb
H}^N_{\ell'}$. Now, although $\ell^*\ge (n-1)/2$, the same argument
as in [BH] still shows that we can still normalize $F^*$ by
composing a certain linear fraction map from ${\mathbb H}^N_{\ell'}$
to ${\mathbb H}^N_{\ell^*,\ell'}$ from the left to get the same
normalization for $F^*$ as in (\ref{norm}) (with $\ell, \tau$ being
replaced by $\ell^*$ and $\tau^*$, respectively). Still write $F^*$
for the normalized $F^*$. Now, the argument in the proof of Theorem
$\ref{main}$  goes through without any change when we replace $\ell,
\tau$ by $\ell^*$ and $\tau^*$, respectively. (Details are left to
the reader.) Hence, one easily see    that in this setting, we have
the statement in Theorem $\ref{main}$ \rm{(b)}. This completes the
proof of Theorem $\ref{main}$.

\end{proof}

\section {Proof of Theorem \ref{main2}}
We present in this section the proof of Theorem \ref{main2}. Our
argument here is partially motivated by the work of Zhang [Zha].

\begin{proof} [Proof of Theorem $\ref{main2}$] As in Section \ref{mainsec} above,
for $p\in \mathbb H^n_\ell$ we let $F_p=(f_p,\phi_p,q_p):=\tau_p^F
\circ F \circ \sigma_0^p$ as in (3.1) in \cite{BH}, where
$\sigma_0^p$ and $\tau_p^F$ are linear automorphisms of the source
and target hyperquadrics such that $F_p(0)=0$. We shall denote by
$M$ a sufficiently small open neighborhood of $0$ in $\mathbb
H^n_{\ell}$. Recall that $F$ is CR transversal to $\mathbb
H^N_{\ell'}$ at $p\in M$ if and only if $\partial
g_p/\partial w(0)\neq 0$. Let $E=\{p\in M: \ F\ \mbox{is not transversal
to $\mathbb H^N_{\ell'}$ at}\ p\}$. The set $E$ is a real-analytic
subvariety near $0$ of $M$. If $E$ contains an open neighborhood of
$0$, then it follows from
Lemma 4.1 in \cite{BH} (see also Theorem 1.1 in \cite{BERtrans}) that $F(U)\subset\mathbb H^N_{\ell'}$, which as mentioned in Remark \ref{gequiv0} implies $g\equiv 0$. Thus, we may assume that the complement $C:=M\setminus E$ is open
and dense. For any $p\in C$, in view of the discussion preceding Theorem \ref{main2}, we can apply  Theorem \ref{main} (a)
or (b), depending on the sign of $\partial g_p/\partial w(0)$, to
conclude that there is a $\tau_p\in Aut_0(\mathbb{H}^N_{\ell'})$
such that
\begin{equation}\Label{tpFp}
(\tau_p \circ F_p)(z,w)=  \left\{
\begin{aligned}
(z_1,\ldots,z_\ell,\psi_p(z,w),z_{\ell+1},\ldots,z_{n-1},\psi_p(z,w),0,\ldots,0
,w)&,\ \text{{\rm if $\frac{\partial g_p}{\partial w}(0)> 0$}}\\
(z_{\ell+1},\ldots,z_{n-1},\psi_p(z,w),z_1,\ldots,z_\ell,\psi_p(z,w),0,\ldots,0
,w)&,\ \text{{\rm if $\frac{\partial g_p}{\partial w}(0)< 0$}}.
\end{aligned}
\right.
\end{equation}
Since $\tau_p\in Aut_0(\mathbb H^N_{\ell'})$, it is of the form
\begin{equation}\Label{taup}
\tau_p(z',w')=\left(\frac{\lambda_p(z'-a_pw')U_p}
{\bigtriangleup_p(z',w')},\frac{\epsilon_p\lambda_p^2w'}
{\bigtriangleup_p(z',w')}\right),
\end{equation}
where $\lambda_p>0$, $\epsilon_p=\pm 1$, $\mathcal
U_pE_{(\ell',N-1)}\overline{\mathcal
U^t_p}=\epsilon_pE_{(\ell',N-1)}$, and $\Delta_p(z',w')$ is a linear
polynomial in $(z',w')$. Note that both mappings on the right hand
side of \eqref{tpFp} are CR transversal (to $\mathbb H^N_{\ell'}$)
at every $p\in \mathbb H^n_{\ell}$. Thus, if $q\in M$ is a point
such that $F_p(q)$ is not on the polar variety of $\tau_p$, i.e.\
$\Delta_p(F_p(q))\neq0$, then $F_p$ is CR transversal at $q$. It
follows that $F$ is CR transversal at $q^*=\sigma^p_0(q)$ and,
hence, $q*\in C$. We conclude that for any $q^*\in E$, the point
$q=(\sigma^p_0)^{-1}(q^*)$ belongs to the polar variety of $\tau_p$.
On the other hand, the last component of $\tau_p\circ F_p$ is
holomorphic near $q$ and, hence, the numerator in the last component
of $\tau_p\circ F_p$ must also vanish at $q$,\ i.e.\
$$
g_p(q)=0.
$$
Now, we have
\begin{align*}
g_p(q)=(g \circ \sigma_0^p)(q)-\overline{g(p)}-2i\langle
(\widetilde{f}\circ\sigma_0^p)(q),\overline{\widetilde{f}(p)}
\rangle_{\ell'}=0,
\end{align*}
or, equivalently,
\begin{align*}
g(q^*)-\overline{g(p)}-2i\langle
\widetilde{f}(q^*),\overline{\widetilde{f}(p)} \rangle_{\ell'}=0.
\end{align*}
or, by complex conjugating the latter,
\begin{equation}\Label{segre}
g(p)-\overline{g(q^*)}-2i\langle
\widetilde{f}(p),\overline{\widetilde{f}(q^*)} \rangle_{\ell'}=0.
\end{equation}
Recall that \eqref{segre} holds for all $q^*\in E$ and all $p\in C$.
Since $C$ is open and dense in $M$, it follows that for each fixed
$q^*\in E$ there is an open neighborhood $U$ of $0$ in $\bC^n$ such
that \eqref{segre} holds for $p\in U$. If we use the notation
$Q'_{q'}$ for the Segre variety of $\mathbb H^N_{\ell'}$ at
$q'=(z'_q,w'_q)$, i.e.\
$$
Q'_{q'}:=\{(z,w)\colon w=\overline{w_q}-2i\left\langle
z',\overline{z'_q}\right \rangle_{\ell'}\},
$$
then \eqref{segre} shows that $F(U)\subset Q'_{F(q^*)}$. In
particular,  if $0\in E$, then $g(z,w)\equiv 0$. This completes the
proof of Theorem \ref{main2}.
\end{proof}

\begin{rem}
Note that in the proof of Theorem \ref{main2} above, we actually
proved that
\begin{equation}
 E=\left\{q^*\in M\colon F(U)\subset Q'_{F(q^*)},\ \text{{\rm for some open neighborhood $U$ of
 $0$ in $
\bC^n$}}\right\}.
\end{equation}
\end{rem}

\section{Examples}

We end this paper with a couple of examples. Our first example shows
that a holomorphic mapping $F$ as in Theorem \ref{main2} can be
non-transversal to $\mathbb H^N_{\ell'}$ at $0$, i.e.\ $\partial
g/\partial w(0)=0$, without sending a full neighborhood $U$ of 0 in
$\bC^n$ into $\mathbb H^N_{\ell'}$, in contrast with the case
$\ell'=\ell$ treated in \cite{BH}. Moreover, the example shows that
there are mappings that are CR transversal to $\mathbb H^N_{\ell'}$
at all points outside a proper real-analytic subvariety of  $\mathbb
H^n_\ell$ without being CR transversal at all points.

\begin{example}\Label{example1}
{\rm Consider the following polynomial mapping $F\colon \bC^5\to
\bC^7$,
\begin{equation}
F(z_1,z_2,z_3,z_4,w):=\left(4z_1z_2,4z_2^2,2z_2(i+w),2z_2(i-w),4z_2z_3,4z_2z_4,0\right).
\end{equation}
It clearly sends 0 to 0. We claim that it also sends $\mathbb H^5_2$
into $\mathbb H^7_3$. Let us write
$$
\rho:=\im w-\left(-|z_1|^2-|z_2|^2+|z_3|^2+|z_4|^2\right)
$$
and
$$
\rho':=\im
w'-\left(-|z'_1|^2-|z'_2|^2-|z'_3|^2+|z'_4|^2+|z'_5|^2+|z'_6|^2\right),
$$
so that $\mathbb H^5_2$ and $\mathbb H^7_3$ are defined by $\rho=0$
and $\rho'=0$, respectively. A straightforward computation, left to
the reader, shows that
$$
\rho'\circ F=4|z_2|^2\rho.
$$
It follows that $F$ does not send a full neighborhood $U$ of $0$ in
$\bC^5$ into $\mathbb H^7_3$ and  $F$ is CR transversal to $\mathbb H^7_3$ precisely
at those $p=(z,w)\in \mathbb H^5_2$ for which $z_2\neq 0$ (see e.g.\ Remark 1.2 in
\cite{BERtrans}). Note that
$\ell'=3<4=2\ell=n-1$. Consequently, both Theorem \ref{main} and \ref{main2} apply. The conclusion of Theorem \ref{main2} is obviously true. Moreover,
at any point where the mapping
$F$ is CR transversal to $\mathbb H^7_3$, it follows from Theorem
\ref{main} that $F$ can be renormalized by composing on the left
with an automorphism of $\mathbb H^7_3$ so as to be of the form
\eqref{normgeq0} or \eqref{normleq0}. (In this particular case, $F$
can be normalized to satisfy either of \eqref{normgeq0} or
\eqref{normleq0} since the signature of $\mathbb H^7_3$ is half its
CR dimension.)}
\end{example}

We conclude this paper by giving an example that shows that Theorems
\ref{main} and \ref{main2} are sharp in the sense that the conclusions fail when the
hypotheses on the signatures $\ell$ and $\ell'$ are not satisfied.

\begin{example}\Label{example2}{\rm
Consider the polynomial mapping $F: \bC^3\to\bC^5$ given by:
$$F(z_1,z_2,w):=\left(z_1+\frac {z_1^2}{2} - {i \over 4} w,z_2-\frac
{z_1z_2}{2},z_1-\frac {z_1^2}{2}+ {i \over 4} w, z_2+\frac {z_1z_2}
{2}z_1w\right). $$ It sends 0 to 0, and we claim that it sends
$\mathbb H^3_1$ to $\mathbb H^5_2$. With the notation
$$ \rho:= \im w -\left(-|z_1|^2+|z_2|^2\right), \ \ \ \rho':=
 \im w' -\left(-|z'_1|^2-|z'_2|^2+|z'_3|^2+|z'_4|^2\right)
$$
for the defining equations of $\mathbb H^3_1$ and $\mathbb H^5_2$,
respectively, we compute that
$$\rho'\circ F=(z_1+\bar
z_1)\rho.
$$
We conclude that $F$ sends $\mathbb H^3_1$ to $\mathbb H^5_2$, as
claimed, and that $F$ is CR transversal to $\mathbb H^5_2$ at $p\in
\mathbb H^3_1$ except on the subvariety of $\mathbb H^3_1$ given by
the intersection with $z_1+\bar z_1=0$. (The transversality at most
points is predicted by Theorem 1.1 of \cite{BERtrans}.) We note that
$\ell'=2=2\ell=n-1$. Also, note the following:

1. The conclusions of Theorem \ref{main} (a) and (b) fail at points where $F$ is CR
transversal to $\mathbb H^5_2$. Indeed, if $F$ could be renormalized
to satisfy either of \eqref{normgeq0} or \eqref{normleq0}, then the
image $F(\bC^3)$ would be contained in a 4-dimensional subspace of
$\bC^5$. It can readily be checked that this is not the case.

2. For any point $p\in\mathbb H^3_1$ at which transversality fails,
there is no open neighborhood $U$ of $p$ in $\bC^3$ such that
$F(U)\subset Q'_{F(p)}$; here, $Q'_{p'}$ denotes the Segre variety
of $\mathbb H^5_2$ at $p'$. Indeed, each Segre variety $Q'_{p'}$ is
a hyperplane in $\bC^5$ and, as above, it can be checked that the
image of $F$ is not contained in any hyperplane. Consequently, the
conclusion of Theorem \ref{main2} also fails. }
\end{example}


\begin{thebibliography}{BER2}

\bibitem [BER1]{BER} M. S. Baouendi, P. Ebenfelt and L. P.
Rothschild:  Real Submanifolds in Complex Space and Their Mappings,
{\it Princeton Math. Ser.} 47, Princeton Univ. Press, Princeton, NJ,
1999


\bibitem[BER2]{BERtrans} M. S. Baouendi, P. Ebenfelt, L. P. Rothschild:
Transversality of holomorphic mappings between real hypersurfaces in
different dimensions. {\it Comm. Anal. Geom.},  15  (2007),  no. 3,
589--611.

\bibitem[BEH]{BEH} M. S. Baouendi, P. Ebenfelt, X. Huang: Super-rigidity
for CR embeddings of real hypersurfaces into hyperquadrics. {\it
Adv. Math.} 219  (2008), 1427-1445.

\bibitem [BH] {BH} M. S. Baouendi and X. Huang,
 Super-rigidity for holomorphic mappings between hyperquadrics with
positive signatures,  {\it J. Diff. Geom.}  69, 379-398, 2005.

\bibitem [CM] {CM} S. S. Chern and J. K. Moser, Real hypersurfaces in complex
manifolds, {\it Acta Math.} 133, (1974), 219-271.

\bibitem [CS] {CS} J. Cima and T. J. Suffridge, A reflection principle with
applications to proper holomorphic mappings, {\it Math Ann. 265}
(1983), 489-500.


\bibitem [DA] {DA} J. P. D'Angelo, Several Complex Variables and the
Geometry of Real Hypersurfaces, CRC Press, Boca Raton, 1993.



\bibitem[EHZ1]{EHZ1}
P. Ebenfelt, X. Huang, D. Zaitsev,  Rigidity of CR-immersions into
spheres. {\em Comm. Anal. Geom.}, {\bf  12} (2004), no. 3, 631--670.

\bibitem[EHZ2]{EHZ2}
P. Ebenfelt, X. Huang, D. Zaitsev,  The equivalence problem and
rigidity for hypersurfaces embedded into hyperquadrics. {\em Amer.
J. Math.}, {\bf 127} (2005),  169--191.



\bibitem[Fa]{Fa} J. Faran, The linearity of proper holomorphic
maps between balls in the low codimension case
, J. Diff. Geom.
\textbf{24}(1986), 15-17.

\bibitem [Fr1] {Fr1} F. Forstneric,  A survey on proper holomorphic mappings,
Proceeding  of Year in SCVs at Mittag-Leffler Institute, Math. Notes
38, Princeton, NJ:
   Princeton University Press , 1992.


\bibitem [Fr2] {Fr2} F. Forstneric,  Extending proper holomorphic mappings of
positive codimension, {\it Invent. Math. 95}, 31-62, 1989.


\bibitem [Ham] {Ham} H. Hamada, Rational proper holomorphic maps from
${\bf B}^n$ into ${\bf B}^{2n}$.  Math. Ann.  331  (2005),  no. 3,
693--711.


 \bibitem [Hu1]{Hu1} X. Huang,  On a linearity problem of proper
 holomorphic mappings between balls in complex spaces of
 different dimensions, J. Diff. Geom.
 \textbf{51}(1999), 13-33.



 \bibitem
[Hu2] {Hu2} X. Huang,
  On a semi-linearity property for holomorphic maps,
  {\it Asian Jour. of Math.} Vol 7 (No. 4), 463-492, 2003. (A special issue in honor of Professor Y-T Siu's 60th birthday).


\bibitem [HJ] {HJ} X. Huang and S. Ji,   On some rigidity problems in Cauchy-Riemann analysis, {\it AMS/IP Studies in Advanced Mathematics}
 Vol. 39, 89-107,
2007.


\bibitem [Mo] {Mok} N. Mok, Metric Rigidity Theorems on Hermitian Locally
Symmetric Manifolds, Series in Pure Mathematics, V. 6, World
Scientific Publishing Co., 1989.






\bibitem [Po] {Po}  H. Poincar\'e, Les fonctions analytiques de deux
variables et la repr\'esentation conforme, {\it  Rend. Circ. Mat.
Palermo}(1907), 185-220.



\bibitem [Zha] {Zha} Y. Zhang, Rigidity and holomorphic Segre transversality for
holomorphic Segre maps, {\it Math. Ann.} Vol. 337 (2) , 457-478,
2007.

\bibitem [We] {We} S.  Webster, On mapping an (n+1)-ball into the complex
space, {\it Pac. J. Math. 81}(1979), 267-272.

\end{thebibliography}
\end{document}